\documentclass[10pt]{amsart}
\usepackage{amsxtra, amsfonts, amsmath, amsthm, amstext, amssymb, amscd, mathrsfs, verbatim, color} 
\usepackage{threeparttable}
\usepackage[ansinew]{inputenc}\usepackage[T1]{fontenc}
\theoremstyle{plain}

\newtheorem{theorem}{Theorem}[section]
\newtheorem{lemma}[theorem]{Lemma}
\newtheorem{definition-theorem}[theorem]{Definition-Theorem}
\newtheorem{proposition}[theorem]{Proposition}
\newtheorem{corollary}[theorem]{Corollary}

\theoremstyle{definition}
\newtheorem{definition}[theorem]{Definition}
\newtheorem{example}[theorem]{Example}
\newtheorem{remark}[theorem]{Remark}
\newtheorem{notation}[theorem]{Notation}
\newcommand \bth[1] { \begin{theorem}\label{t#1} }
\newcommand \ble[1] { \begin{lemma}\label{l#1} }

\newcommand \bpr[1] { \begin{proposition}\label{p#1} }
\newcommand \bco[1] { \begin{corollary}\label{c#1} }
\newcommand \bde[1] { \begin{definition}\label{d#1}\rm }
\newcommand \bex[1] { \begin{example}\label{e#1}\rm }
\newcommand \bre[1] { \begin{remark}\label{r#1}\rm }

\newcommand \bnota[1] {\begin{notation}\label{n#1}\rm }
\newcommand {\ele} { \end{lemma} }

\newcommand {\epr} { \end{proposition} }
\newcommand {\eco} { \end{corollary} }
\newcommand {\ede} { \end{definition} }
\newcommand {\eex} { \end{example} }
\newcommand {\ere} { \end{remark} }
\newcommand {\enota} { \end{notation} }

\def \Id { {\mathrm{Id}} }

\DeclareMathOperator \Hom { {\mathrm{Hom}} }

\DeclareMathOperator \Ind { {\mathrm{Ind}} }
\DeclareMathOperator \Irr { {\mathrm{Irr}} }
\DeclareMathOperator \Res { {\mathrm{Res}} }

\DeclareMathOperator \sgn { { \mathrm{sgn}}}

\DeclareMathOperator \im { { {\mathrm im}}}

\DeclareMathOperator \Seg { { { \mathrm Seg}}}
\DeclareMathOperator \triv { { {\mathrm triv}}}

\begin{document}
\setlength{\baselineskip}{1.2\baselineskip}
\title[Dirac cohomology for degenerate affine Hecke-Clifford algebras]
{Dirac cohomology for degenerate affine Hecke-Clifford algebras}
\author[Kei Yuen Chan]{Kei Yuen Chan}
\address{
Korteweg-de Vries Institute for Mathematics, Universiteit van Amsterdam}
\email{K.Y.Chan@uva.nl}

\begin{abstract} In this paper, we study the Dirac cohomology theory on a class of algebraic structure. The main examples of the algebraic structure are the degenerate affine Hecke-Clifford algebra of type $A_{n-1}$ by Nazarov and of classical types by Khongsap-Wang. The algebraic structure contains a remarkable subalgebra, which usually refers to Sergeev algebra for type $A_{n-1}$.

We define an analogue of the Dirac operator for those algebraic structures. A main result is to relate the central characters of modules of those algebras with the central characters of modules of the Sergeev algebra via the Dirac cohomology. The action of the Dirac operator on certain modules is also computed. Results in this paper could be viewed as a projective version of the Dirac cohomology of the degenerate affine Hecke algebra.
\end{abstract}

\maketitle

\section{Introduction}

Throughout this paper, we work over the ground field $\mathbb{C}$. Let $W$ be a Weyl group. It is well-known that $W$ admits a non-trivial central extension
\[  1 \rightarrow \mathbb{Z}_2 \rightarrow \widetilde{W} \rightarrow W \rightarrow 1 ,\]
where $\widetilde{W}$ is a distinguished double cover of $W$. The projective representations of $W$ are linear representations of $\widetilde{W}$ which do not factor through $W$. Those representations over $\mathbb{C}$ have been has been known for a long time from the work of Schur, Morris, Read, Stembridge and others \cite{Mo1, Mo2, Re, Sc, St}.


The degenerate affine Hecke-Clifford algebra for type $A_{n-1}$ (see Definition \ref{def hca}) was introduced by Nazarov \cite{Na} to study the Young's symmetrizers of the projective representations of $S_n$. The degenerate affine Hecke-Clifford algebra for other classical types was later constructed by Khongsap-Wang \cite{WK}. Those algebras could be viewed as the projective counterpart of the degenerate affine Hecke algebra of Lusztig.




The purpose of this paper is to establish Dirac cohomology theory for those classes of algebras. We first single out the algebraic structure (see Section \ref{s hw}) that is necessary to prove several important results for the Dirac cohomology, and then we show that the degenerate affine Hecke-Clifford algebras considered in \cite{Na} and \cite{WK} satisfy that algebraic structure. Our approach is an analogue of the one recently developed for degenerate affine Hecke algebras by Barbasch, Ciubotaru and Trapa \cite{BCT} (also see a recent extension by Ciubotaru \cite{Ci2}).

In more detail, let $\mathcal H_W$ be the associative algebra with certain important properties (see Definitions \ref{def hw property star} and \ref{def hw property star2}). The algebra $\mathcal H_W$ contains a remarkable subalgebra, namely $\mathrm{Seg}(W)$ (see again Definition \ref{def hw property star}), which is is the same as the Sergeev algebra when $W$ is of type $A_{n-1}$.)
The Dirac type element in $\mathcal H_W$ is defined in analogue of the one in \cite{BCT} and have some nice properties. In specific examples of $\mathcal H_W$ in Section \ref{s dirac hcl}, the Dirac type element can be viewed as the square root of a certain Casmir type element (Theorem \ref{thm d sq}).

For an $\mathcal{H}_W$-module $(\pi, X)$, the Dirac cohomology is defined as
\[ H_D(X)=\ker \pi(D)/(\ker \pi(D) \cap \im \pi(D)),\] which is a $\Seg(W)$-module. One of our main results (Theorem \ref{thm vog conj hw}) says that if $X$ is irreducible and $H_D(X)$ is nonzero, then any irreducible $\Seg(W)$-module in $H_D(X)$ determines the central character of $X$. This is an analogue to a statement for Harish-Chandra modules called Vogan's conjecture \cite{HP}. A key step in the proof of Theorem \ref{thm vog conj hw} is to establish a canonical algebra homomorphism from the center of $\mathcal{H}_W$ to the center of $\Seg(W)$ (Theorem \ref{thm rel centers}). In the case of the degenerate affine Hecke-Clifford algebra of type $A_{n-1}$, this homomorphism is shown to map onto the even elements of the center of $\Seg(W)$ via the study of the Dirac cohomology on some modules (Corollary \ref{cor zeta onto}). The homomorphism indeed agrees with another natural map arising from the Jucys-Murphy type elements (see more detail in Remark \ref{rmk surjective homo}), and hence the property of surjectivity has already been covered in the result of \cite{Ru}.

As Dirac cohomology in other settings (see, for example \cite{HP}), one may apply the Dirac operator and Dirac cohomology developed in this paper to study the representation theory of $\mathcal{H}_W$. More precisely, the action of the Dirac operator provides information about the $\Seg(W)$-module structure and central characters of some $\mathcal{H}_W$-modules (see Corollary \ref{cor par dirac inq} and Theorem \ref{thm Vogan conjecture HC}).

We provide evidences that the Dirac cohomology can be useful in the representation theory by computing the action of the Dirac operators in several cases. In Section \ref{s ex vanishing Dirac}, we consider some basic modules for all classical types and show that the Dirac operator acts identically zero on those modules. Those modules for type $A_{n-1}$ are constructed and studied by Hill-Kujawa-Sussan \cite{HKS}. In Section \ref{s spec}, we go further for type $A_{n-1}$ and compute the action of the Dirac type element $D$ on more interesting modules. We show that the Dirac cohomology of those examples does not vanish, and this indeed coincides with the expectation from the case of the degenerate affine Hecke algebra in \cite{BCT}. While some computations can also be done for other classical types, the picture is more complete for type $A_{n-1}$ to date.



This paper is organized as follows. In Section \ref{s prelmin}, we review some properties of superalgebras. In Section \ref{s hw}, we define a certain algebraic structure $\mathcal H_W$ and develop the Dirac cohomology theory for $\mathbb{H}_W$. We provide examples of $\mathcal H_W$ in Section \ref{s dirac hcl} and compute the square of the Dirac operator. In Section \ref{s ex vanishing Dirac} and Section \ref{s spec}, we consider the Dirac cohomology for some particular modules. In Section \ref{s seg alg}, we review properties of Sergeev algebra which is needed for the computation of Section \ref{s spec}.


\subsection{Acknowledgment} The author would like to thank Dan Ciubotaru and Peter Trapa for the suggestion of this topic and many useful discussions. He also thanks Professor Weiqiang Wang for his interest in the work and pointing out the reference \cite{Wa}. The author would also like to thank referees for useful suggestions and comments, and also thank one of the referees to point out the reference \cite{Ru}.

\section{Preliminaries} \label{s prelmin}
\subsection{Notation for modules}

In this paper, all the algebras are associative with an unit over $\mathbb{C}$. Let $\mathcal A$ be an algebra. An $\mathcal A$-module is denoted $(\pi, X)$ or simply $X$, where $X$ is a vector space and $\pi$ is the map defining the action of $\mathcal A$ on $X$. For $a \in \mathcal A$ and $x \in X$, the action of $a$ on $x$ is written by $\pi(a)x$ or $a.x$.

Let $\mathcal B$ be a subalgebra of $\mathcal A$.   Define $\Ind^{\mathcal A}_{\mathcal B}$ to be the induction functor, i.e.
\[   \Ind^{\mathcal A}_{\mathcal B} Y = \mathcal A \otimes_{\mathcal B} Y ,\]
where $Y$ is a $\mathcal B$-module. The left adjoint functor of $\Ind^{\mathcal A}_{\mathcal B}$ is the restriction functor denoted $\Res^{\mathcal A}_{\mathcal B}$.

\subsection{Superalgebras and supermodules} \label{ss prelim superalg}
A super vector space $V$ is a $\mathbb{Z}_2$-graded vector space $V=V_0 \oplus V_1$. A super vector subspace $W$ of $V$ is a subspace of $V$ such that $W=(W \cap V_0) \oplus (W \cap V_1)$. We say an element $a$ in $ V_0$ (resp. $ V_1$) has even (resp. odd) degree, denoted $\deg(v)=0$ (resp. $\deg(v)=1$).

 A superalgebra $\mathcal A$ is an algebra with a super vector space structure $\mathcal A=\mathcal A_0 \oplus \mathcal A_1$ and $\mathcal A_i \mathcal A_j \subseteq \mathcal A_{i+j}$ for $i,j \in \mathbb Z_2$. A subalgebra $\mathcal C$ of a superalgebra $\mathcal A$ is said to be a super subalgebra of $\mathcal A$ if $\mathcal C=(\mathcal A_0 \cap \mathcal C)\oplus (\mathcal A_1 \cap \mathcal C)$. A super ideal $\mathcal I$ of a superalgebra $\mathcal A$ is an ideal of $\mathcal A$ such that $\mathcal I=(\mathcal A_0 \cap \mathcal I) \oplus (\mathcal A_1 \cap \mathcal I)$.

For superalgebras $\mathcal A$ and $\mathcal B$, a superalgebra homomorphism from $\mathcal A$ to $\mathcal B$ is an algebra homomorphism with $f(\mathcal A_i) \subset \mathcal B_i$ for $i \in \mathbb{Z}_2$.

For superalgebras $\mathcal A$ and $\mathcal B$, the super tensor product of $\mathcal A$ and $\mathcal B$, denoted $\mathcal A \widetilde{\otimes} \mathcal B$ is a superalgebra isomorphic to $\mathcal A \otimes \mathcal B$ as vector spaces with the multiplication determined by:
\[    (a \otimes b)(a' \otimes b')=(-1)^{\deg(b)\deg(a')}(aa'\otimes bb') ,\]
where $a,a' \in \mathcal A$ and $b,b' \in \mathcal B$ are homogeneous elements.


Let $\mathcal A$ be a superalgebra. An $\mathcal A$-supermodule $X$ is an $\mathcal A$-module with a super vector space structure $X=X_0 \oplus X_1$ and the property that $\mathcal A_i . X_j \subseteq X_{i+j}$, where $i, j \in \mathbb Z_2$. A supersubmodule  $Y$ of an $\mathcal A$-supermodule $X$ is a submodule of $X$ such that $Y=(X_0 \cap Y) \oplus (X_1 \cap Y)$. An $\mathcal A$-supermodule $X$ is irreducible if there is no proper non-zero supersubmodule of $X$.

For an $\mathcal A$-supermodule $X=X_0 \oplus X_1$, define a map $\delta: X \rightarrow X$ such that $\delta(v)=v$ if $v \in X_0$ and $\delta(v)=-v$ if $v \in X_1$.

Let $\mathrm{Mod}_{\mathrm{sup}}(\mathcal A)$ be the category of $\mathcal A$-supermodules. The morphisms in the category $\mathrm{Mod}_{\mathrm{sup}}(\mathcal A)$ are the even homomorphisms between $\mathcal A$-supermodules. Let $\Pi: \mathrm{Mod}_{\mathrm{sup}}(\mathcal A) \rightarrow \mathrm{Mod}_{\mathrm{sup}}(\mathcal A)$ be a parity change functor. That means for an $\mathcal A$-supermodule, $\Pi(M)$ and $M$ are isomorphic as $\mathcal A$-modules, but have opposite $\mathbb Z_2$-grading.

\subsection{Relations between irreducible supermodules and irreducible modules} \label{ss rel sup and ord}

Let $\mathcal A=\mathcal A_0 \oplus \mathcal A_1$ be a superalgebra. Given an irreducible $\mathcal A$-module $(\pi, Y)$, we construct a supermodule as follows. Let  $(\overline{\pi}, \overline{Y})$ be an irreducible $\mathcal A$-module such that $\overline{Y}$ is identified with $Y$ as vector spaces and the $\mathcal A$-action on $\overline{Y}$ is determined for any homogenous element $a \in \mathcal A$ and for $v \in Y$  by
\[ \overline{\pi}(a)v= (-1)^{\deg(a)}\pi(a)v .\]
Let $(\pi_{X_Y}, X_Y)$ be an $\mathcal A$-supermodule such that $X_Y=Y \oplus \overline{Y}$ as vector spaces and the action of $\mathcal A$ on $X_Y=Y\oplus \overline{Y}$ is as:
$\pi_{X_Y}(a)(v,\overline{v})=(\pi(a)v, \overline{\pi}(a)\overline{v})$. Let $(X_Y)_0=\left\{ (v,\overline{v}) \in X_Y: v=\overline{v} \right\}$ and let $(X_Y)_1=\left\{ (v, \overline{v}) \in X_Y : v=-\overline{v} \right\}$. It is elementary to check $X_Y=(X_Y)_0 \oplus (X_Y)_1$ is an $\mathcal A$-supermodule.

\begin{lemma} \label{lem crit irr general}
Let $Y$ be an irreducible $\mathcal A$-module. Let $X_Y=Y\oplus \overline{Y}$ be an $\mathcal A$-supermodule with the supermodule structure described above. Then
\begin{enumerate}
\item[(1)]  $X_Y$ is an irreducible $\mathcal A$-supermodule  if and only if $Y$ and $\overline{Y}$ are non-isomorphic as $\mathcal A$-modules.
\item[(2)] If $Y$ and $\overline{Y}$ are isomorphic as $\mathcal A$-modules, then there is a supermodule structure on $Y$.
\end{enumerate}
\end{lemma}

\begin{proof}
For (1), we first prove that if $X_Y$ is an irreducible $\mathcal A$-supermodule, then $Y$ and $\overline{Y}$ are not isomorphic as $\mathcal A$-modules. Suppose instead there exists an $\mathcal A$-module isomorphism $f:Y \rightarrow \overline{Y}$ and we will derive a contradiction. Recall that $\overline{Y}$ is identified with $Y$ as vector spaces and thus there exists a natural vector space isomorphism $\theta: \overline{Y} \rightarrow Y$ such that $(-1)^{\deg(a)}\pi(a)\theta=\theta \overline{\pi}(a)$ for any homogenous $a \in \mathcal A$.  Then $\theta \circ f$ satisfies the  property that for any homogenous element $a \in \mathcal A$,
\[ \pi(a)(\theta \circ f)(x)=(-1)^{\deg(a)} (\theta \circ f)(\pi(a)x) \quad .\] Then the map $ (\theta \circ f)^2$ is an $\mathcal A$-module automorphism of $Y$. Thus, by Schur's lemma and a suitable normalization, we may assume $(\theta \circ f)^2$ is an identity map. Then as vector spaces
\[  Y= \ker(\theta \circ f-\Id) \oplus \ker(\theta \circ f+\Id). \]
For $\epsilon=0,1$, let
\[   \mathrm{Ker}_{\epsilon}=\left\{ (v,(-1)^{\epsilon}v) \in X_Y : v \in \ker(\theta \circ f-(-1)^{\epsilon} \Id) \right\}. \]
Then it is straightforward to verify $\mathrm{Ker}_0 \oplus \mathrm{Ker}_1 \subset X_Y$ gives a proper super submodule of $X_Y$.

We now prove if $Y$ and $\overline{Y}$ are not isomorphic as $\mathcal A$-modules, $X_Y$ is an irreducible $\mathcal A$-supermodule. Suppose instead that there exists a proper super submodule $M$ of $X_Y$ and we will get a contradiction. Let $M^i=\left\{ v \in Y : (v, (-1)^iv) \in M \cap (X_Y)_i \right\}$ for $i \in \mathbb{Z}_2$, which are regarded as vector subspaces of $Y$. We first see that $M^0 \cap M^1=0$. Otherwise, there exists some nonzero $v \in Y$ such that $(v,v) \in M$ and $(v,-v) \in M$, and so $(v,0), (0,v) \in M$. The irreducibility of $Y$ and $\overline{Y}$ implies $M=X_Y$, contradicting $M$ is proper. Furthermore the irreducibility of $Y$ implies $Y=M^0 \oplus M^1$ (as vector spaces). Define a map $f: (\pi, Y) \rightarrow (\overline{\pi}, \overline{Y})$ determined by $f(v)=(-1)^iv$ for $v \in M^i$ ($i \in \mathbb{Z}_2$). One can check $f$ is an $\mathcal A$-module isomorphism and so this gives a contradiction.

We now consider (2). By (1), $X_Y$ is not an irreducible $\mathcal A$-supermodule. Let $X'$ be an irreducible super submodule of $X_Y$. Then by the construction of $X_Y$, $X'$ is isomorphic to $Y=\overline{Y}$ as $\mathcal A$-modules. Then this gives a supermodule structure to $Y$.
\end{proof}

We can also start with an irreducible $\mathcal A$-supermodule and decompose it into irreducible $\mathcal A$-module(s).

\begin{lemma} \label{lem crit irred r}
Let $X$ be an irreducible $\mathcal A$-supermodule. Let $\delta$ be a linear automorphism on $X$ such that $\delta(v)=(-1)^{i}v$ for $v \in X_{i}$ ($i=0,1$). If $X$ is not an irreducible $\mathcal A$-module, then there exists an irreducible $\mathcal A$-submodule $Y$ of $X$ such that
\begin{itemize}
\item[(1)] $\delta(Y)$ is also an $\mathcal A$-submodule of $X$ and $\delta(Y)=\overline{Y}$; and
\item[(2)] $Y$ and $\delta(Y)$ are non-isomorphic $\mathcal A$-modules; and
\item[(3)] $X=Y \oplus \delta(Y)$ as $\mathcal A$-modules.
\end{itemize}

\end{lemma}

\begin{proof}
(1) follows from $a.\delta(v)=(-1)^{\deg(a)}\delta(a.v)$ for any homogenous element $a \in \mathcal A$ and $v \in Y$. (2) and (3) are (a reformulation of) \cite[Lemma 2.3]{BK}.
\end{proof}

\begin{lemma} \label{lem unique structure}
Let $X$ and $X'$ be irreducible $\mathcal A$-supermodules. If $X$ and $X'$ are isomorphic as $\mathcal A$-modules, then $X$ and $X'$ are isomorphic, up to applying the functor $\Pi$, as $\mathcal A$-supermodules.
\end{lemma}

\begin{proof}
Suppose $X$ and $X'$ are also irreducible $\mathcal A$-modules. Then $X_0, X_1, X_0', X_1'$ are irreducible $\mathcal A_0$-modules. Then either $X_0=X_0'$ or $X_0=X_1'$ as $\mathcal A_0$-modules. Then either $X \cong X'$ or $X \cong \Pi(X')$ as $\mathcal A$-supermodules.

Suppose $X$ is not an irreducible $\mathcal A$-module. Let $X=Y \oplus \delta(Y)$ and $X'=Y' \oplus \delta(Y')$ be the decomposition of $X$ into $\mathcal A$-modules as in Lemma \ref{lem crit irred r}. Without loss of generality, we may assume $Y=Y'$ as $\mathcal A$-modules. Let $f: Y \rightarrow Y'$ be an $\mathcal A$-module isomorphism. Then $f$ also induces an $\mathcal A$-module isomorphism $\overline{f}: \delta(Y) \rightarrow \delta(Y')$ such that $\overline{f}=\delta \circ f \circ \delta$. Then one can show the map $f \oplus \overline{f}$ is an $\mathcal A$-supermodule isomorphism by checking the map preserves grading. In particular, we also have $\Pi(X)=X$ as $\mathcal A$-supermodules in this case.
\end{proof}

Let $\Irr(\mathcal A)$ (resp. $\Irr_{\mathrm{sup}}(\mathcal A)$) be the set of irreducible $\mathcal A$-modules (resp. irreducible $\mathcal A$-supermodules). Let $\sim$ be the equivalence relation on $\Irr(\mathcal A)$: $Y\sim Y'$ if and only if $Y=Y'$ or $Y=\overline{Y'}$. Let $\sim_{\Pi}$ be the equivalence relation on $\Irr_{\mathrm{sup}}(\mathcal A)$: $X \sim_{\Pi} X'$ if and only if $X=X'$ or $X=\Pi(X')$.
\begin{proposition} \label{prop bij sup and ord}
There is a natural bijection
\[  \Irr_{\mathrm{sup}}(\mathcal A)/\sim_{\Pi}\ \longleftrightarrow \ \Irr(\mathcal A)/\sim .\]

\end{proposition}

\begin{proof}
Lemmas \ref{lem crit irr general} and \ref{lem unique structure} define a map from $\Irr(\mathcal A)/\sim$ to $\Irr_{\mathrm{sup}}(\mathcal A)/\sim_{\Pi}$. Lemma \ref{lem crit irred r} defines a map in the opposite direction. The two maps are inverse to each other by Lemma \ref{lem unique structure}.

\end{proof}

\subsection{Central characters of supermodules}

For a superalgebra $\mathcal A$, let $Z(\mathcal A)$ be the center of $\mathcal A$. Note that $Z(\mathcal A)$ is a super subalgebra of $\mathcal A$. Recall that $Z(\mathcal A)_0$ is the set of even elements in $Z(\mathcal A)$.

\begin{proposition} \label{prop schur}
Let $X$ be an irreducible $\mathcal A$-supermodule. For $z \in Z(\mathcal A)_0$, $z$ acts on $X$ by the multiplication of a scalar.
\end{proposition}
 \begin{proof}
If $X$ is an irreducible $\mathcal A$-module, then the statement follows from (ordinary) Schur's lemma (for this case). If $X$ is not an irreducible $\mathcal A$-module, then we can decompose $X=Y \oplus \delta(Y)$ as $\mathcal A$-modules as in Lemma \ref{lem crit irred r}. Then $z$ acts on the two modules $Y$ and $\delta(Y)$ by scalars, denoted $\lambda$ and $\lambda'$ respectively. Then for $v \in Y$,
\[  z.(v+\delta(v)) = \frac{\lambda+\lambda'}{2}(v+\delta(v))+\frac{\lambda-\lambda'}{2}(v-\delta(v)) \]
Note that $\delta(v+\delta(v))=v+\delta(v)$ and so $v+\delta(v) \in X_0$, and similarly $v-\delta(v) \in X_1$. Then since $z$ is of even degree, $\lambda=\lambda'$.

\end{proof}

By Proposition \ref{prop schur}, we can define the following:

\begin{definition} \label{def central character}
Let $\mathcal A$ be a superalgebra. Let $(\pi, X)$ be an irreducible $\mathcal A$-supermodule. Define the central character $\chi_{\pi}$ to be the map from $Z(\mathcal A)_0$ to $\mathbb C$ such that $\chi_{\pi}(z)$ is the scalar of $z$ acting on $X$.
\end{definition}

The central character defined above is only for even elements in the center of a superalgebra. However, the central character indeed determines the action of odd elements in the center in the following sense:

\begin{proposition}
Let $z \in Z(\mathcal A)_1$. Let $X$ be an irreducible $\mathcal A$-supermodule. If $X$ is also an irreducible $\mathcal A$-module, then $z$ acts by zero on $X$. If $X$ is not an irreducible $\mathcal A$-module, then $z$ acts on the two irreducible $\mathcal A$-submodules of $X$ by two distinct scalars $ \sqrt{\lambda}$ and $-\sqrt{\lambda}$, where $\lambda$ is the scalar that $z^2 \in Z(\mathcal A)_0$ acts on $X$.
\end{proposition}
\begin{proof}
 For (1), suppose $X$ is an irreducible $\mathcal A$-module. Then by Schur's Lemma,  $z$ acts on $X$ by a scalar denoted by $\lambda$. Meanwhile by Lemmas \ref{lem crit irr general} and \ref{lem unique structure}, $X =\overline{X}$ as $\mathcal A$-modules. This implies $z$ also acts by $-\lambda$ on $X$ as $z$ is an odd element. Hence $\lambda=0$.

Now suppose $X$ is not an irreducible $\mathcal A$-module. Then $z^2$ is an even element in the center and hence acts by a scalar, denoted $\lambda$. Then $z$ acts on the irreducible $\mathcal A$-submodules of $X$ by scalars $ \sqrt{\lambda}$ and $-\sqrt{\lambda}$.
\end{proof}




\section{Dirac cohomology for $\mathcal{H}_{W}$} \label{s hw}


\subsection{$\mathcal H_W$ and a Dirac type element in $\mathcal H_W$} \label{ss HW}


Fix a real reflection group $W$. Let $V$ be a representation of $W$. Fix a $W$-invariant inner product on $V$.  Let $\left\{ a_1, \ldots, a_n \right\}$ be an orthogonal basis for $V$.
\begin{definition} \label{def hw property star}
An associative algebra $\mathcal{H}_W=\mathcal H_W(V)$ is said to have property (*) if it satisfies the following properties. First $\mathcal{H}_W$ is an algebra generated by symbols $f_w$ ($w \in W$), $c_i$ ($i=1, \ldots, n$) and $a_i$ ($i=1, \ldots, n$) such that the map from $\mathbb{C}[W]$ to $\mathcal H_W$ sending $w$ to $f_w$ is an injection and the algebra has a natural basis of elements having the form $ a_1^{k_1}\ldots a_n^{k_n}c_1^{\epsilon_1}\ldots c_n^{\epsilon_n}f_w$  ($k_1, \ldots, k_n$ non-negative integers, $w \in W$, $\epsilon_i =0$ or $1$). Again we shall write $w$ for $f_w$ for simplicity. Let $\Seg(W)$ be the subalgebra of $\mathcal H_W$ generated by all $w \in W$ and $c_i$ ($i=1,\ldots, n$). Furthermore, the generators of $\mathcal H_W$ satisfy the following relations:
\begin{eqnarray}
 \label{rel star 0}     wa_iw^{-1} & = & w(a_i) \\
     \left[a_i,a_j \right]c_ic_j  &\in & \Seg(W)  \quad \mbox{ for $i \neq j$} \\
 \label{rel star 1}      c_ja_i &=& a_ic_j \quad \mbox{ for $i \neq j$ } \\
  \label{rel star 2}       c_ia_i &=& - c_ia_i  \\
 \label{rel star 3}        c_ic_j&=&-c_jc_i \quad \mbox{ for $i \neq j$ }  \quad \mbox{ and } \quad c_i^2=-1 \\
   \label{rel star 4}      wc_i&=&w(c_i)w .
 \end{eqnarray}
 Here $w(a_i)$ is the action of $w$ on $V$. Furthermore, we identify the linear space spanned by $c_i$ with $V$ via the map $a_i \mapsto c_i$ and henc there is a natural action of $W$ on $c_i$, and $w(c_i)$ represents such action of $w$ on $c_i$. Indeed the algebra generated by the those $c_i$ is isomorphic to the Clifford algebra on the vector space $V$, and the subalgebra $\mathrm{Seg}(W)$ is the smash product of the Clifford algebra and the group algebra of $W$. 

$\mathcal H_W$ has a superalgebra structure with $\deg(c_i)=1$, $\deg(a_i)=\deg(w)=0$ ($i=1, \ldots, n$ and $w \in W$).
 \end{definition}


In the rest of this section, $\mathcal H_W$ denotes an algebra satisfying the property (*). Define a Dirac type element $D$ in $\mathcal H_W$:
\begin{eqnarray} \label{eqn Dirac form}
  D = \sum_{i=1}^n a_ic_i .
\end{eqnarray}

The following two properties will be used several times:

\begin{lemma} \label{lem seg comm D}
\begin{enumerate}
\item[(1)] $wD=Dw$ for any $w \in W$;
\item[(2)] $c_iD=-Dc_i$ for any $i$.
\end{enumerate}
\end{lemma}

\begin{proof}
(1) follows from the fact that $\left\{ a_i \right\}$ forms an orthogonal basis and property (\ref{rel star 0}). (2) follows from the properties  (\ref{rel star 1}),  (\ref{rel star 2}) and (\ref{rel star 3}).
\end{proof}





Two homogenous elements $h_1, h_2 \in \mathcal H_W$ are said to supercommute if $h_1h_2=(-1)^{\deg(h_1)\deg(h_2)}h_2h_1$ for any homogenous $w \in \Seg(W)$.
\begin{definition} \label{def hw property star2}
The algebra $\mathcal H_W$ with the property (*) is said to satisfy the property (**) if for any $h \in \mathcal H_W$ such that $h$ supercommutes with elements in $\Seg(W)$, $D^2h-hD^2=0$.
\end{definition}

In the next section, we shall give examples which satisfy the algebraic structure in Definitions \ref{def hw property star} and \ref{def hw property star2}. From now on, assume that $\mathcal H_W$ satisfies the properties (*) and (**).

\subsection{Relation between central characters for $\mathcal H_W$ and $\Seg(W)$}

Let $d: \mathcal H_W \rightarrow \mathcal H_W$,
\[d(h)=Dh-(-1)^{\deg(h)}hD. \]
A relation between $Z(\mathcal H_W)_0$ and $Z(\Seg(W))_0$ is the following:

\begin{theorem} \label{thm rel centers}
For any $z \in Z(\mathcal H_W)_0$, there exists a unique element $\widetilde{z} \in Z(\Seg(W))_0$ such that
\[      z-\widetilde{z} \in \im d .\]
Let $\zeta: Z(\mathcal H_W)_0 \rightarrow Z(\Seg(W))_0$ be the map that $\zeta(z)$ is such unique element $\widetilde{z}$ in $Z(\Seg(W))_0$. Then $\zeta$ is an  algebra homomorphism.
\end{theorem}

Our main result in this paper is the following which says the central character of an $\mathcal H_W$-supermodule $X$ is determined by the central characters of irreducible $\Seg(W)$-supermodules in the Dirac cohomology $H_D(X)$. Here $H_D(X)$ is defined in the theorem.

\begin{theorem} \label{thm vog conj hw}
Let $\mathcal H_W$ be an algebra satisfying property (*) (Definition \ref{def hw property star}) and property (**) (Definition \ref{def hw property star2}). Let $(\pi, X)$ be an irreducible  $\mathcal H_W$-supermodule with the central character $\chi_{\pi}$ (Definition \ref{def central character}). Let the Dirac cohomology $H_D(X)$ of $X$ be
\[  H_D(X) = \ker \pi(D)/(\ker \pi(D) \cap \im \pi(D)) . \]
Then $H_D(X)$ has a natural $\Seg(W)$-module structure. Let $(\sigma, U)$ be an irreducible $\Seg(W)$-module with the central character $\chi_{\sigma}$ (Definition \ref{def central character}) such that  \[\Hom_{\Seg(W)}(U, H_D(X)) \neq 0.\]
Let $\zeta: Z(\mathcal H_W)_0 \rightarrow Z(\Seg(W))_0$ be the map in Theorem \ref{thm rel centers}. Let $\chi^{\sigma} : Z(\mathcal H_W)_0 \rightarrow \mathbb{C}$,
\begin{eqnarray} \label{eqn central character sigma}
   \chi^{\sigma}(z) = \chi_{\sigma}(\zeta(z)) .
\end{eqnarray}
 Then $\chi_{\pi}= \chi^{\sigma}$.
\end{theorem}

Since $wD=Dw$ and $c_iD=-Dc_i$ by Lemma \ref{lem seg comm D}, $\ker \pi(D)$ and $\ker \pi(D) \cap \im \pi(D)$ are invariant under the action of $\Seg(W)$. Thus $H_D(X)$ has a natural $\Seg(W)$-module structure from the $\mathcal H_W$-module structure. The proofs of Theorems \ref{thm rel centers} and \ref{thm vog conj hw} are given at the end of the next subsection. Theorem \ref{thm vog conj hw} directly follows from Theorem \ref{thm rel centers}. Readers who only want to know how Theorem \ref{thm rel centers} implies Theorem \ref{thm vog conj hw} may jump to the end of the next subsection.


\subsection{Proof of Theorems \ref{thm rel centers} and \ref{thm vog conj hw}}
The proofs of the theorems basically follow from the ideas of proofs in \cite[Chapter 3]{HP} and \cite[Section 4]{BCT}. We provide some technical details for this specific case.

 Let $S^{ \leq j}(V)$ be the vector space of polynomials of $x_1, \ldots, x_n$ with degree less than or equal to $j$. Let $\mathcal H_{W}^j$ be the vector space spanned by elements of the form
\[    \left\{    pw : w \in \Seg(W) , p \in S^{\leq j}(V) \right\} . \]
Note that $\mathcal H_W^0 \subseteq \mathcal H_W^1 \subseteq \ldots$ gives a filtration for $\mathcal H_W$. Define
\[   \overline{\mathcal H}_W^r = \mathcal H_W^{r}/ \mathcal H_W^{r-1} ,\]
for $r=0,1,\ldots$ and $\mathcal H_W^{-1}=0$. Let $\overline{\mathcal H}_W= \oplus_{j=0}^{\infty} \overline{\mathcal H}_W^j$. Note that $\overline{\mathcal H}_W$ has a natural superalgebra structure from $\mathcal H_W$.

 Let $d_j:\overline{\mathcal H}_W^j \rightarrow \overline{\mathcal H}_W^{j+1}$ be the map induced from $d$ and let $\overline{d} = \oplus_{j=0}^{\infty} d_j$. For any element $h \in \mathcal H_W$, we still write $h$ for its corresponding element in $\overline{\mathcal H}_W$. Let $b_i=a_ic_i$ ($i=1,\ldots, n$). Let $\mathcal B$ be the supersubalgebra of $\overline{\mathcal H}_W$ generated by all $b_i$. Note that $\overline{d}(\mathcal B) \subset \mathcal B$. Let $\overline{d}'$ be the restriction of $\overline{d}$ to $\mathcal B$.

In the following lemmas, one can see $\ker \overline{d}'$, $\im \overline{d}'$, $\ker \overline{d}$, $(\ker d \cap \im d)^{\Seg(W)}$ and so on are supersubspaces by using the fact that $D$ is an homogenous element.

\begin{lemma} \label{lem d on Y}
As supersubspaces of $\mathcal B$,
\[ \ker \overline{d}' = \im \overline{d}' \oplus \mathbb{C} .\]
Here $\mathbb{C}$ is regarded as the $\mathbb{C}$-subalgebra of $\mathcal B$ generated by $1$.
\end{lemma}

\begin{proof}
 Note that any element in $\mathcal B$ can be uniquely written as a linear combination of elements of the form $p b_{i_1}b_{i_2}\ldots b_{i_r}$ for $0<i_1 <\ldots <i_r \leq n$ and $p \in \mathbb{C}[b_1^2, \ldots, b_n^2]$. Note that $D=\sum_{i=1}^n b_i$. Using the relations $b_ib_j=-b_jb_i  $ (in $\mathcal B$) for $i \neq j$ and $b_i^2b_j=b_jb_i^2$ (in $\mathcal B$) for any $i,j$, one can see the action of $\overline{d}'$ is determined by
\[\overline{d}'(pb_{i_1}b_{i_2}\ldots b_{i_r})=2 \sum_{k=1}^r (-1)^{k-1}b_{i_k}^2pb_{i_1}\ldots \widehat{b_{i_k}} \ldots b_{i_r},\]
where  $p \in \mathbb{C}[b_1^2, \ldots, b_n^2]$.

In order to apply the known cohomology of the Koszul complex, we identify $\mathcal B$ with $\mathbb{C}[x_1,\ldots, x_n] \otimes \wedge^{\bullet} \mathbb{C}^n$ as vector spaces, where $\wedge^{\bullet} \mathbb{C}^n$ is the exterior algebra, via the linear isomorphism $\eta$ from $\mathbb{C}[x_1,\ldots, x_n] \otimes \wedge^{\bullet}\mathbb{C}^n$ to $\mathcal B$ determined by
\[ \eta : p(x_1,\ldots,x_n)\otimes e_{i_1}\wedge \ldots \wedge e_{i_k} \mapsto p(b_1^2,\ldots, b_n^2)b_{i_1}\ldots b_{i_k}, \]
where $\left\{ e_1, \ldots, e_n \right\}$ is the standard basis of $\mathbb{C}^n$.
Then, by the above description of the action of $\overline{d}'$, the map $\eta^{-1} \circ \overline{d}' \circ \eta$ is a multiple of the differential map in the standard Koszul resolution. Then the result follows from the well-known cohomology of the Koszul resolution.

\end{proof}

\begin{proposition} \label{prop vog skew}
As supersubspaces of $\overline{\mathcal H}_W$,
\[  \ker \overline{d}=\im \overline{d} \oplus \Seg(W) .\]
\end{proposition}

\begin{proof}
By the property (*) of $\mathcal H_W$,  $a_{1}^{m_1}a_2^{m_2}\ldots a_n^{m_n}c_{1}^{\epsilon_{1}}\ldots c_n^{\epsilon_n}w$ ($m_i \in \mathbb{Z}_{\geq 0}$, $\epsilon_i \in \left\{ 0, 1 \right\}$ and $w \in W$) form a basis for ${\overline{\mathcal H}}_W$. Then $b_{1}^{m_1}b_2^{m_2}\ldots b_n^{m_n}c_{1}^{\epsilon_{1}}\ldots c_n^{\epsilon_n}w$ ($m_i \in \mathbb{Z}_{\geq 0}$, $\epsilon_i \in \left\{ 0, 1 \right\}$ and $w \in W$) also form a basis for ${\overline{\mathcal H}}_W$. Then as linear vector spaces, we may identify $\overline{\mathcal H}_W$ with $\mathcal B \otimes \Seg(W)$ via the following map:
\[ b_{1}^{m_1}b_2^{m_2}\ldots b_n^{m_n}c_{1}^{\epsilon_{1}}\ldots c_n^{\epsilon_{n}}w \mapsto b_{1}^{m_1}\ldots b_n^{m_n} \otimes c_{1}^{\epsilon_{1}}\ldots c_n^{\epsilon_n}w . \]
For any $h \in \overline{\mathcal H}_W$, $\overline{d}(hw)=\overline{d}(h)w$ for $w \in W$ and $\overline{d}(hc_i)=\overline{d}(h)c_i$. Then the map $\overline{d}$ in $\mathcal{H}_W$ is the same as $\overline{d}' \otimes \Id$ in $\mathcal B \otimes \Seg(W)$ under the above identification.
Then by Lemma \ref{lem d on Y}, one has
\[  \ker \overline{d}=\ker (\overline{d}' \otimes \Id) =(\ker \overline{d}') \otimes \Seg(W) =(\im \overline{d}' \oplus \mathbb{C} ) \otimes \Seg(W) = \im \overline{d} \oplus \Seg(W) .\]

\end{proof}

For any subspace $H$ of $\mathcal H_W$, define $H^{\Seg(W)}$ to be the set of all element supercommuting with elements in $\Seg(W)$. If we view $\Seg(W)$ as a subalgebra of $\overline{\mathcal H}_W$, we could similarly define $\overline{H}^{\Seg(W)}$ for any subspace $\overline{H}$ of $\overline{\mathcal H}_W$. Proposition \ref{prop vog skew} implies the following:

\begin{corollary} \label{cor vog skew}
As supersubspaces of $\overline{\mathcal H}_W$,
\[ (\ker  \overline{d})^{\Seg(W)}=(\im \overline{d})^{\Seg(W)} \oplus Z(\Seg(W)) . \]
\end{corollary}



\begin{lemma} \label{lem ker decomp Hw}
As supersubspaces of $\mathcal H_W$,
\[ (\ker d)^{\Seg(W)} = (\ker d \cap \im d)^{\Seg(W)}  \oplus Z(\Seg(W)). \]
\end{lemma}

\begin{proof}
It is clear that $Z(\Seg(W))$ and $(\ker d \cap \im d)^{\Seg(W)}$ are subspaces of $(\ker d)^{\Seg(W)}$ and thus $(\ker d \cap \im d)^{\Seg(W)}  \oplus Z(\Seg(W)) \subset (\ker d)^{\Seg(W)} $. We will prove another inclusion by induction on the degree of filtration of an element in $(\ker d)^{\Seg(W)}$.

Let $h$ be an element in $(\ker d)^{\Seg(W)}$ such that $h \in \mathcal H_W^i$ and $h \not\in \mathcal H_W^{i-1}$ for some $i$. When $i=0$, $\mathcal H_W^0=\Seg(W)$ and so the statement is clearly true. Now assume $i > 0$. Let $\overline{h}$ be the image of $h$ in $\overline{\mathcal H}_W^i$. Then by Corollary \ref{cor vog skew},
$\overline{h} =\overline{d}(\overline{h}_0)$ for some unique $\overline{h}_0$ in $\overline{\mathcal H}_W^{i-1}$ such that $\overline{d}(\overline{h}_0) \in (\overline{\mathcal H}_W^{i})^{\Seg(W)}$.
For any representative $h_0' \in \mathcal H_W^{i-1}$ of $\overline{h}_0$, let
\[ h_0=\frac{1}{2^n|W|}\sum_{k=1}^n \sum_{i_1< \ldots <i_k} \sum_{w \in W}(-1)^k (c_{i_1}\ldots c_{i_k})wh_0'w^{-1}(c_{i_1}\ldots c_{i_k})^{-1}. \]
By the uniqueness of the element $\overline{h}_0$, $\overline{h}_0$ supercommutes with any element in $\Seg(W)$. This implies $h_0$ is also a representative of $\overline{h}_0$. Furthermore, $h_0$ supercommutes with elements in $\Seg(W)$ and $d(h_0) \in (\mathcal H_W^i)^{\Seg(W)}$. By the property (**), $d^2(h_0)=0$ and so $d(h-d(h_0))=0$. By the induction hypothesis, $h-d(h_0)\in (\im d)^{\Seg(W)} \oplus Z(\Seg(W))$. Hence, we also have $h \in (\im d)^{\Seg(W)} \oplus Z(\Seg(W))$ since $d(h_0) \in (\im d)^{\Seg(W)}$. This completes the proof.

\end{proof}


\begin{lemma} \label{lem center homo}
$(\ker d)^{\Seg(W)}$ is a super subalgebra of $\mathcal{H}_W$ and $(\ker d \cap \im d)^{\Seg(W)}$ is a two sided super ideal of $(\ker d)^{\Seg(W)}$.
\end{lemma}

\begin{proof}
Let $z_1, z_2 \in (\ker d)^{\Seg(W)}$. Then $d(z_i)=0$ and so $Dz_i=\delta(z_i)D$. Now $d(z_1z_2)=D(z_1z_2)-\delta(z_1z_2)D=\delta(z_1z_2)D-\delta(z_1z_2)D=0$. Hence $z_1z_2 \in (\ker d)^{\Seg(W)}$. Hence $(\ker d)^{\Seg(W)}$ is a subalgebra of $\mathcal{H}_W$.

We next show that $(\ker d \cap \im d)^{\Seg(W)}$ is a two sided ideal of $(\ker d)^{\Seg(W)}$. Let $z \in (\ker d)^{\Seg(W)}$ and $z' \in (\ker d \cap \im d)^{\Seg(W)}$. We have to show $zz', z'z \in  (\ker d \cap \im d)^{\Seg(W)}$. Write $z'=Dh-\delta(h)D$ for some $h \in \mathcal H_W$. Since $d(z)=Dz-\delta(z)D=0$,
\[zz'=zDh-z\delta(h)D=D\delta(z)h-z\delta(h)D=D\delta(z)h-\delta(z\delta(h))D \in \im d .\]
We also proved in the beginning that $zz' \in \ker d$ and thus $zz' \in  (\ker d \cap \im d)^{\Seg(W)}$. The proof for $z'z\in  (\ker d \cap \im d)^{\Seg(W)}$ is similar.

\end{proof}

{\noindent}
{\it Proof of Theorem \ref{thm rel centers} }
Since $z \in Z(\mathcal H_W)_0 \subset \ker d^{\Seg(W)}$, by Lemma \ref{lem ker decomp Hw}, there exists a unique $\widetilde{z} \in Z(\Seg(W))$ such that $z-\widetilde{z} \in  (\ker d \cap \im d)^{\Seg(W)} \subset \im d$. Note that $\widetilde{z}$ is in $Z(\Seg(W))_0$ since the decomposition in Lemma \ref{lem ker decomp Hw} is between super vector spaces. Hence we have a map $\zeta: Z(\mathcal H_W)_0 \rightarrow Z(\Seg(W))_0$.

It remains to prove that $\zeta$ is an algebra homomorphism. To see $\zeta$ is an algebra map, let $z_i \in Z(\mathcal H_W) \subset  \ker d^{\Seg(W)}$ ($i=1,2$). Write $z_i=\zeta(z_i)+h_i$ for some $h_i \in   (\ker d \cap \im d)^{\Seg(W)}$. Then $z_1z_2=\zeta(z_1)\zeta(z_2)+\zeta(z_1)h_2+\zeta(z_2)h_1+h_1h_2$. By Lemma \ref{lem center homo}, $z_1z_2-\zeta(z_1)\zeta(z_2) \in  (\ker d \cap \im d)^{\Seg(W)}$. Thus $\zeta(z_1z_2)=\zeta(z_1)\zeta(z_2)$. This completes the proof.


\bigskip
{\noindent}
{\it Proof of Theorem \ref{thm vog conj hw}}
By our hypothesis, there exists a non-zero element $v \in H_{D}(X)$ such that $v$ is in the isotypic component $U$ of $H_{D}(X)$. Let $\widetilde{v}$ be a representative of $v$ in $\ker \pi(D)$. Now by Theorem \ref{thm rel centers} for any $z \in Z(\mathcal H_W)_0$, $z-\zeta(z)=Da-\delta(a)D$ for some $a \in \mathcal H_W$. Then $\pi(z-\zeta(z)) \widetilde{v}=\pi(Da-\delta(a)D)\widetilde{v}=\pi(Da)\widetilde{v} \in \im \pi(D)$.  On another hand, $\pi(z-\zeta(z))\widetilde{v}= \chi_{\pi}(z)\widetilde{v}-(\chi_{\sigma}(\zeta(z))\widetilde{v}+\widetilde{v}')$ for some $\widetilde{v}' \in \ker \pi(D) \cap \im \pi(D)$ and so $(\chi_{\pi}(z)-\chi_{\sigma}(\zeta(z)))\widetilde{v} \in \im \pi(D)$. We also have $(\chi_{\pi}(z)-\chi_{\sigma}(\zeta(z)))\widetilde{v} \in \ker \pi(D)$ as $\widetilde{v} \in \ker \pi(D)$. Thus  $\chi_{\pi}(z)\widetilde{v}-\chi_{\sigma}(\zeta(z))\widetilde{v} \in \im \pi(D) \cap \ker \pi(D)$. Since we choose $v \neq 0$, we can only have $\chi_{\pi}(z)=\chi_{\sigma}(\zeta(z))=\chi^{\sigma}(z)$. This completes the proof.

\section{Examples of $\mathcal{H}_W$ and their Dirac cohomology theory} \label{s dirac hcl}

Let $W$ be a classical Weyl group and let $R=R(W)$ be the root system associated to $W$. Let $\mathbf{k}: R \rightarrow \mathbb{C}$ be a function such that $\mathbf{k}(\alpha_1)=\mathbf{k}(\alpha_2)$ if $\alpha_1=w(\alpha_2)$ for some $w \in W$. We shall write $\mathbf{k}_{\alpha}$ for $\mathbf{k}(\alpha)$. For any $\alpha \in R$, let $s_{\alpha}$ be the simple reflection associated to $\alpha$.

 Let $e_1, \ldots, e_n$ be the standard basis of $\mathbb{R}^n$. Let $\langle , \rangle$ be the inner product on $\mathbb{R}^n$ such that $\langle e_i, e_j \rangle=\delta_{ij}$.


\subsection{Type $A_{n-1}$} \label{ss comm rel}

\begin{notation}
Set $W=W(A_{n-1})$ to be the Weyl group of type $A_{n-1}$. The root system $R(A_{n-1})$ of type $A_{n-1}$ is the set
\[ R(A_{n-1})= \left\{ e_i-e_j  : 1 \leq i \neq j \leq n \right\}  .\]
Fix a set $R^+$ of positive roots
\[ R^+(A_{n-1})=\left\{ e_i-e_{j}  : 1 \leq i<j \leq n \right\} .\]
We usually write $\alpha >0$ for $\alpha \in R^+(A_{n-1})$ and write $\alpha <0$ for $-\alpha \in R^+(A_{n-1})$.
The set of simple roots $\Delta$ is
\[ \left\{ e_i-e_{i+1} : i=1,\ldots, n-1 \right\} . \]

Since there is only one $W$-orbit for $R(A_{n-1})$, we simply write $\mathbf{k}$ for $\mathbf{k}_{\alpha}$ for any $\alpha \in R(A_{n-1})$.
For $i \neq j$, let
\begin{align} \label{eqn positive roots} \alpha_{ij} = \left\{  \begin{array}{c} e_i-e_j \quad \mbox{ if $i <j$ }  \\
                                              e_j-e_i \quad \mbox{ if $i >j$ }
                                         \end{array}
                              \right. .
      \end{align}
Thus $\alpha_{ij}$ is always a positive root.

For a root $\alpha \in R(A_{n-1})$, let $s_{\alpha}$ be the corresponding simple reflection in $W(A_{n-1})$. For simplicitly, set $s_{ij}=s_{\alpha_{ij}}$.
\end{notation}

\begin{definition} \label{def hca} \cite{Na}
The degenerate affine Hecke-Clifford algebra for type $A_{n-1}$, denoted $\mathbb{H}^{Cl}_{W(A_{n-1})}$, is the associative algebra with an unit generated by the symbols $\left\{ x_i \right\}_{i=1}^n$, $\left\{ c_i \right\}_{i=1}^n$ and $\left\{ f_w :  w \in W(A_{n-1}) \right\}$ determined by the following properties:
\begin{enumerate}
\item[(1)] the map from the group algebra $\mathbb{C}[W(A_{n-1})]=\oplus_{w \in W(A_{n-1})} \mathbb{C}w$ to $\mathbb{H}^{Cl}_{W(A_{n-1})}$ given by $w \mapsto f_w$ is an algebra injection;
\item[(2)] $x_ix_j=x_jx_i$ for all $i,j$;
\item[(3)] $x_ic_j=c_jx_i$ for $i \neq j$ and $x_ic_i=-c_ix_i$ for all $i$;
\item[(4)] $c_ic_j=-c_jc_i$ for $i \neq j$ and $c_i^2=-1$ for all $i$;
\item[(5)] $f_wc_i=c_{w(i)}f_w$ for $w \in W(A_{n-1})$ and and for all $i$
\item[(6)] $f_{s_{i,i+1}}x_i-x_{i+1}f_{s_{i,i+1}}=\mathbf{k}(-1+c_ic_{i+1})$ for all $i=1,\ldots, n-1$ and $f_{s_{i,i+1}}x_j=x_jf_{s_{i,i+1}}$ for all $i,j$ with $|i-j|>1$.
\end{enumerate}
We later simply write $w$ for $f_w$. The algebra has a superalgebra structure with $\deg(c_i)=1$, $\deg(w)=0$ for $w \in W(A_{n-1})$, and $\deg(x_i)=0$.

\end{definition}

For $i \neq j$, define $c_{\alpha_{ij}}$ as
\begin{align} \label{eqn Clifford An} c_{\alpha_{ij}} = \left\{ \begin{array}{c} \frac{\sqrt{2}}{2}(c_i -c_j) \quad \mbox{ if $i <j$ } \\
                                                \frac{\sqrt{2}}{2}(c_j-c_i) \quad \mbox{ if $j <i$ }
                                 \end{array} \right. .
\end{align}
Let $\widetilde{s}_{\alpha_{ij}}=\widetilde{s}_{ij}=s_{ij}c_{\alpha_{ij}}$.

The superalgebra $\mathbb{H}^{Cl}_{W(A_{n-1})}$ admits a PBW type basis:

\begin{proposition} \cite[Theorem 14.2.2]{Kl} \label{prop basis HC}
The set
\[   \left\{ x_{1}^{m_1}\ldots x_n^{m_n}c_1^{\epsilon_1}\ldots c_n^{\epsilon_n}w : m_1,\ldots, m_n \in \mathbb{Z}_{\geq 0}, \epsilon_1,\ldots, \epsilon_n \in \left\{ 0, 1\right\}, w \in W(A_{n-1}) \right\} \]
forms a basis for $\mathbb{H}^{Cl}_{W(A_{n-1})}$.
\end{proposition}

The main statement of this subsection is Proposition \ref{prop isom hcl}, which says $\mathbb{H}^{Cl}_{W(A_{n-1})}$ satisfies property (*) defined in Definition \ref{def hw property star}.

Let $\widetilde{s}_{\alpha}=s_{\alpha}c_{\alpha}$. For later convenience, we also set $\widetilde{s}_{ij}=\widetilde{s}_{\alpha_{ij}}=s_{\alpha_{ij}}c_{{\alpha}_{ij}}$,  $y_i=x_ic_i$, $y_i'=y_i+\frac{\sqrt{2}}{2}\sum_{i \neq j} \widetilde{s}_{i,j}$ and $x_i'=-y_i'c_i$. Note that $\mathbb{C}[W(A_{n-1})]^-$ embeds into $\mathbb{H}^{Cl}_{W(A_{n-1})}$ via the map $\widetilde{t}_{\alpha} \mapsto \widetilde{s}_{\alpha}$.

The notation $y_i'$ and $x_i'$ will be used to define the Dirac type element in $\mathbb{H}^{Cl}_{W(A_{n-1})}$ and are inspired by the setting in the degenerate affine Hecke algebra in \cite{BCT}.

\begin{lemma} \label{lem simple rel}
\begin{itemize}
\item[(1)]  $c_iy_j=-y_jc_i$ for any $i,j$;
\item[(2)]  $\widetilde{s}_{ij}c_k=-c_k\widetilde{s}_{ij}$ for any $i,j,k$ with $i \neq j$;
\item[(3)]  $c_iy_j'=-y_j'c_i$ for any $i,j$;
\item[(4)]  For $\alpha \in R^+$ and $w \in S_n$, $w\widetilde{s}_{\alpha}w^{-1}=\widetilde{s}_{w(\alpha)}$ if $w(\alpha)>0$, and $w\widetilde{s}_{\alpha}w^{-1}=-\widetilde{s}_{-w(\alpha)}$ if $w(\alpha)<0$.
\end{itemize}
\end{lemma}

The above lemma is elementary. We skip the proof.

 We shall use the natural permutation of $W(A_{n-1})$ on the set $\left\{1,\ldots, n \right\}$ below.

\begin{lemma} \label{lem comm form}
Let $w \in W(A_{n-1})$. Then
\[  wy_iw^{-1}-y_{w(i)} =\sqrt{2}\mathbf{k} \sum_{\beta>0, w^{-1}(\beta)< 0, \langle \beta, w(e_i) \rangle \neq0}   \widetilde{s}_{\beta} . \]
In particular, for $\alpha >0$,
\[   \widetilde{s}_{\alpha}y_i\widetilde{s}_{\alpha}^{-1}+y_{s_{\alpha}(i)}= -\sqrt{2}\mathbf{k}\sum_{\beta>0, s_{\alpha}^{-1}(\beta)<0, \langle \beta, s_{\alpha}(e_i) \rangle \neq 0} \widetilde{s}_{\beta}.  \]
\end{lemma}

\begin{proof}
For $w \in W(A_{n-1})$, define $l(w)=|\left\{ e_i-e_j \in R^+(A_{n-1}) : w(e_i-e_j)<0 \right\}|$. When $l(w)=1$, $w=s_{\alpha}$ for some $\alpha \in \Delta$. We consider three cases. When $\langle e_i, \alpha \rangle =0$, it is easy to see $s_{\alpha}y_is_{\alpha}-y_i=0$. Now consider the case $\langle e_i, \alpha \rangle =1 $. In this case, we have
\begin{eqnarray*}
  s_{\alpha}y_is_{\alpha} &=& s_{\alpha}x_ic_is_{\alpha}  \\
	              &=& x_{i+1}c_{i+1}+\mathbf{k}(-1+c_ic_{i+1})c_{i}s_{\alpha} \\
								&=& x_{i+1}c_{i+1}+\mathbf{k}(-c_i+c_{i+1})s_{\alpha} \\
								&=& x_{i+1}c_{i+1}+\mathbf{k}s_{\alpha}(c_i-c_{i+1}) \\
								&=& y_{i+1}+\sqrt{2}\mathbf{k}\widetilde{s}_{\alpha} .
\end{eqnarray*}
For $\langle e_i, \alpha \rangle =-1$, by using $s_{\alpha}\widetilde{s}_{\alpha}s_{\alpha}=-\widetilde{s}_{\alpha}$ and the computation in the case $\langle e_i, \alpha \rangle =1 $, we have
\[   s_{\alpha}y_{i+1}s_{\alpha}=y_i +\sqrt{2}\mathbf{k}\widetilde{s}_{\alpha} . \]
We now use an induction on $l(w)$. Assume $l(w)=k$ for some $k>1$. Write $w=s_{\alpha}w'$ for some simple reflection $s_{\alpha}$ and $w' \in W(A_{n-1})$ with $l(w')=k-1$. Set $\epsilon=1$ if $\langle \alpha, w(e_i) \rangle \neq 0$ and $\epsilon=0$ otherwise. Then
\begin{eqnarray*}
 wy_iw^{-1} & =& s_{\alpha}w'y_iw'^{-1}s_{\alpha} \\
                            &=& s_{\alpha} y_{w'(i)}s_{\alpha}+\sqrt{2}\mathbf{k}\sum_{\beta>0, w'^{-1}(\beta)<0, \langle \beta, w'(e_i) \rangle \neq 0} s_{\alpha}\widetilde{s}_{\beta}s_{\alpha} \quad \mbox{ (induction hypothesis)} \\
                            &=&  y_{s_{\alpha}w'(i)}+\epsilon\sqrt{2}\mathbf{k}\widetilde{s}_{\alpha}+\sqrt{2}\mathbf{k}\sum_{\beta>0, w'^{-1}(\beta)<0, \langle \beta, w'(e_i) \rangle \neq 0}\widetilde{s}_{s_{\alpha}(\beta)}
														 \quad \mbox{ (calculation for $l(w)=1$)} \\
														 &=&  y_{s_{\alpha}w'(i)}+ \epsilon\sqrt{2}\mathbf{k}\widetilde{s}_{\alpha}+\sqrt{2}\mathbf{k}\sum_{\beta>0, w'^{-1}(\beta)<0, \langle s_{\alpha}(\beta), s_{\alpha}w'(e_i) \rangle \neq 0}\widetilde{s}_{s_{\alpha}(\beta)}  \\
                            &=&  y_{w(i)} + \sqrt{2}\mathbf{k}\sum_{\beta>0, w^{-1}(\beta)< 0, \langle \beta, w(e_i) \rangle \neq0}   \widetilde{s}_{\beta}
\end{eqnarray*}
This proves the first assertion. The second assertion follows from the first one with the equation that \[ \widetilde{s}_{\alpha}y_i\widetilde{s}_{\alpha}^{-1}=s_{\alpha}c_{\alpha}y_i(-c_{\alpha}s_{\alpha})=s_{\alpha}(c_{\alpha}^2)y_is_{\alpha}=-s_{\alpha}y_is_{\alpha}. \]
\end{proof}

\begin{lemma} \label{lem comm rel 1}
For $i \neq j$, $[x_i',x_j']c_ic_j=y_i'y_j'+y_j'y_i'    \in \Seg_n$.
\end{lemma}
\begin{proof}
\begin{eqnarray*}
& & y_i'y_j'+y_j'y_i'  \\
&=& (y_i+\frac{\sqrt{2}}{2}\mathbf{k}\sum_{k \neq i} \widetilde{s}_{i,k})(y_j+\mathbf{k}\frac{\sqrt{2}}{2}\sum_{l\neq j}\widetilde{s}_{l,j})+(y_j+\frac{\sqrt{2}}{2}\mathbf{k}\sum_{l \neq j} \widetilde{s}_{l,j})(y_i+\frac{\sqrt{2}}{2}\mathbf{k}\sum_{k\neq i}\widetilde{s}_{i,k})\\
&=& y_iy_j+y_jy_i+\frac{\sqrt{2}}{2}\mathbf{k}\left(\sum_{k \neq i} \widetilde{s}_{i,k} y_j+y_j\sum_{k \neq i}\widetilde{s}_{i,k}+y_i \sum_{l \neq j} \widetilde{s}_{l,j}+\sum_{l \neq j} \widetilde{s}_{l,j}y_i\right) \\
& & \quad +\frac{1}{2}\mathbf{k}^2\left(\sum_{k \neq i}\widetilde{s}_{i,k}\sum_{l \neq j}\widetilde{s}_{l,j} +\sum_{l \neq j}\widetilde{s}_{l,j}\sum_{i \neq k}\widetilde{s}_{i,k}\right) \\
&=& \frac{\sqrt{2}}{2}\mathbf{k}\left(\sum_{k \neq i} \widetilde{s}_{i,k} y_j+y_j\sum_{k \neq i}\widetilde{s}_{i,k}+y_i \sum_{l \neq j} \widetilde{s}_{l,j}+\sum_{l \neq j} \widetilde{s}_{l,j}y_i\right)
+\frac{1}{2}\mathbf{k}^2\left(\sum_{l \neq j} \sum_{k \neq i} \widetilde{s}_{l,j}\widetilde{s}_{i,k} + \sum_{l\neq j} \sum_{k \neq i} \widetilde{s}_{i,k}\widetilde{s}_{l,j} \right)\\
\end{eqnarray*}
By Lemma \ref{lem comm form}, the term $\frac{\sqrt{2}}{2}\left(\sum_{k \neq i} \widetilde{s}_{i,k} y_j+y_j\sum_{k \neq i}\widetilde{s}_{i,k}+y_i \sum_{l \neq j} \widetilde{s}_{l,j}+\sum_{l \neq j} \widetilde{s}_{l,j}y_i\right)$ is in $\Seg(W(A_{n-1}))$. This completes the proof.


\end{proof}

\begin{lemma} \label{lem comm rel 2}
\begin{enumerate}
\item[(1)] $wx_i'w^{-1}=x_{w(i)}'$;
\item[(2)] $c_ix_i'=-x_i'c_i$ and $c_jx_i'=x_i'c_j$ for $i \neq j$.
\end{enumerate}
\end{lemma}

\begin{proof}
For (1), it suffices to show when $w=s_{\alpha}$ for some $\alpha \in \Delta$. Fix an $i$. By the definition of $x_i'$, it suffices to show $s_{\alpha}y_i's_{\alpha}=y_{s_{\alpha}(i)}'$.
We consider two cases. In the case that $\langle e_i, \alpha \rangle =0$, $s_{\alpha}(\alpha_{i,j})>0$ for any $j \neq i$. Then $s_{\alpha}\widetilde{s}_{i,j}s_{\alpha}=\widetilde{s}_{i,s_{\alpha}(j)}$ for any $j \neq i$. Thus, the last equality in Lemma \ref{lem comm form} becomes
\[ s_{\alpha}y_i's_{\alpha}^{-1} = y_i+\frac{\sqrt{2}}{2}\mathbf{k}\sum_{j \neq i} \widetilde{s}_{i,s_{\alpha}(j)} =y_i'
\]
In the case that $\langle e_i, \alpha \rangle \neq 0$, let $k = i-1$ or $i+1$ such that $\alpha=\alpha_{i,k}$. Then, by Lemmas \ref{lem simple rel}(4) and \ref{lem comm form},
\begin{eqnarray*}
 s_{\alpha}y_i's_{\alpha}^{-1}&=& y_{s_{\alpha}(i)}+\sqrt{2}\mathbf{k}\widetilde{s}_{\alpha}-\frac{\sqrt{2}}{2}\mathbf{k}\widetilde{s}_{\alpha}+\frac{\sqrt{2}}{2}\mathbf{k}\sum_{j \neq i,k}\widetilde{s}_{k,j} \\
 &=&  y_k+\frac{\sqrt{2}}{2}\mathbf{k}\sum_{j \neq k} \widetilde{s}_{k,j} \\
 &=& y_k'
\end{eqnarray*}
For (2), it is straightforward from Lemma \ref{lem simple rel} and $y_i'=x_i'c_i$.
\end{proof}
\begin{remark}
The subalgebra of $\mathbb{H}^{Cl}_{W(A_{n-1})}$ generated by the elements $y_i$ and $\widetilde{s}_{i,j}$ is the degenerate spin affine Hecke algebra of type $A_{n-1}$ defined in \cite[Section 3.3]{Wan}. (Other classical types for the degenerate spin affine Hecke algebra are established in \cite[Section 4]{WK}.) The degenerate spin affine Hecke algebra can be regarded as more elementary analogue of the degenerate affine Hecke algebra, and the notions of $y_i'$ can be regarded as the Drinfield presentation \cite{Dr} under the analogue.

\end{remark}

\begin{proposition} \label{prop isom hcl}
The degenerate affine Hecke-Clifford algebra $\mathbb{H}^{Cl}_{W(A_{n-1})}$ satisfies the property (*) in Definition \ref{def hw property star}.
\end{proposition}

\begin{proof}
We set $W$ in Definition \ref{ss HW} equal to $W(A_{n-1})$ and  set $a_i$ in Definition \ref{ss HW} to be $x_i'$. Using Lemmas \ref{lem comm rel 1} and \ref{lem comm rel 2}, one can verify relations (\ref{rel star 0}) to (\ref{rel star 4}) in Definition \ref{ss HW}. By Proposition \ref{prop basis HC} and expressions of $x_i'$, $(x_{1}')^{m_1}\ldots (x_n')^{m_n}c_1^{\epsilon_1}\ldots c_n^{\epsilon_n}w$ ($m_1, \ldots, m_n \in \mathbb{Z}$, $\epsilon_1, \ldots, \epsilon_n \in \left\{ 0,1 \right\}$, $w \in W(A_{n-1})$) form a basis for $\mathbb{H}^{Cl}_{W(A_{n-1})}$. These verify the property (*).
\end{proof}

\subsection{Type $B_n$} \label{ss Dirac type Bn}
 For type $B_n$, we modify the original definition in \cite{WK}. More precisely, the algebra we considered in Definition \ref{def dahc other cl} is a deformation of the algebra in \cite{WK}. It is not hard to do a similar modification for type $A_{n-1}$. The main reason for this modification is to construct an explicit module in the next section, which cannot be done in the original definition of \cite{WK} (by our approach). Considering lack of existing literature for the representation theory of the degenerate affine Hecke-Clifford algebra for other classical types, such example may be interesting and important.

\begin{notation} \label{notation Bn}
Let $W=W(B_n)$ be the Weyl group of type $B_n$. Let the set $R(B_n)$ of roots for type $B_n$ be
\[  R(B_n) = \left\{ \pm e_i \pm e_j : 1 \leq i < j \leq n \right\} \cup \left\{ \pm e_i : i=1, \ldots, n \right\} .
\]
The roots $\pm e_i \pm e_j$ ($i \neq j$) are long while the roots $\pm e_i$ are short. Fix a set $R^+(B_n)$ of positive roots:
\[  R^+(B_n) = \left\{  e_i \pm e_j : 1 \leq i < j \leq n \right\} \cup \left\{ e_i : i=1, \ldots, n \right\} .
\]
The set $\Delta$ of simple roots is
\[ \left\{ e_i-e_{i+1} : i=1,\ldots, n-1 \right\} \cup \left\{ e_n \right\} . \]
For $i \neq j >0$, define $\alpha_{ij}$ as in (\ref{eqn positive roots}), define $\alpha_{i,-j}=e_i+e_j$ and define $\alpha_i=e_i$. We also define $s_{ij}=s_{\alpha_{i,j}}$, $s_{i,-j}=s_{\alpha_{i,-j}}$ and $s_i=s_{\alpha_i}$.

We have a natural embedding $R(A_{n-1}) \subset R(B_n)$. and a natural embedding $W(A_{n-1}) \subset W(B_n)$ (i.e. the group $W(A_{n-1})$ being the group generated by $s_{i,i+1}$ for $i=1,\ldots, n-1$).

\end{notation}

\begin{definition} \label{def dahc other cl}
Let $N_{B_n} \in \mathbb{C}$. Let $\mathbb{H}^{Cl}_{W(B_n)}=\mathbb{H}^{Cl}_{W(B_n)}(\mathbf{k}, N_{B_n})$ be the associative unital algebra generated by the symbols $\left\{ x_i \right\}_{i=1}^n$, $\left\{ c_i \right\}_{i=1}^n$ and $\left\{ f_w :  w \in W(B_n) \right\}$ subject to the relations of (3), (4), (5), (6) in Definition \ref{def hca} and additionally,
\begin{enumerate}
\item the map from the group algebra $\mathbb{C}[W(B_n)]=\oplus_{w \in W(B_n)} \mathbb{C}w$ to $\mathbb{H}^{Cl}_{W(B_n)}$ given by $w \mapsto f_w$ is an algebra injection;
\item $f_{s_n}c_n=-c_nf_{s_n}$ and $f_{s_n}c_i=c_if_{s_n}$ for $i \neq n$,
\item
\begin{align*}
 f_{s_n}x_n+x_nf_{s_n}&=-\sqrt{2}\mathbf{k}_{\alpha_n},  \\
  f_{s_n}x_j-x_jf_{s_n}&=0 \quad \mbox{ for $j \neq n$ }.
 \end{align*}
\item   $x_ix_j-x_jx_i =N_{B_n} c_jc_i$ for $i \neq j$.
\end{enumerate}
When $N_{B_n}=0$, $\mathbb{H}^{Cl}_{W(B_n)}(\mathbf{k}, N_{B_n})$ coincides with the degenerate affine Hecke-Clifford algebra of type $B_n$ in \cite[Definition 3.9]{WK}. We shall again simply write $w$ for $f_w$.

For $N_{B_n}\neq 0$, while $x_i$ and $x_j$ does not commute for $i \neq j$, we still have $x_i^2x_j=x_jx_i^2$. The algebra $\mathbb{H}^{Cl}_{W(B_n)}$ hence still has some nice properties such as the commutation relations with intertwining operators (but we do not need this in this paper).
\end{definition}

For $i \neq j >0$, define $c_{\alpha_{ij}}$ as in (\ref{eqn Clifford An}) and define
\[  c_{\alpha_{i,-j}} =  \frac{\sqrt{2}}{2}(c_i +c_j)   .
\]
Set $\widetilde{s}_{i,-j}=s_{i,-j}c_{\alpha_{i,-j}}$. We also set  $\widetilde{s}_{\alpha}=\widetilde{s}_i=s_ic_i$.

Since we have modified the original definition of the degenerate affine Hecke-Clifford algebra for type $B_n$ in \cite{WK}, we will give a proof for the existence of the PBW type basis.

\begin{proposition} \label{prop PBW Bn}
The set
\[   \left\{ x_{1}^{m_1}\ldots x_n^{m_n}c_1^{\epsilon_1}\ldots c_n^{\epsilon_n}w : m_1,\ldots, m_n \in \mathbb{Z}_{\geq 0}, \epsilon_1,\ldots, \epsilon_n \in \left\{ 0, 1\right\}, w \in W(B_n) \right\} \]
forms a basis for $\mathbb{H}^{Cl}_{W(B_n)}$.

\end{proposition}
\begin{proof}
We follow the argument in \cite[Theorem 3.2.2]{WK}. We consider the algebra $\widetilde{\mathbb{H}}$ generated by $\left\{\overline{x}_i \right\}$, $\left\{ \overline{c}_i \right\}$ and $\left\{ \overline{s}_{i,i+1} \right\}_{i=1}^{n-1} \cup \left\{ \overline{s}_n \right\}$ subject to the relations  (3), (4), (5), (6) in Definition \ref{def hca}  and the relation (2) (but not (1)) in Definition \ref{def hca Dn} (with a trivial replacement of notations). We resolve the minimal ambiguities according to the Bergman's diamond lemma \cite{Be}. For example, we may consider an ordering $\overline{s} <\overline{c}_n <\ldots<\overline{c}_1< \overline{x}_n < \ldots <\overline{x}_1$, where $s$ is any simple reflection in $W(B_n)$. This induces a semigroup ordering on $\langle \overline{x}_i, \overline{c}_i, \overline{s} \rangle$ ($i=1,\ldots,n$ and $s$ runs for all simple reflections) from the length of words and the lexicographical ordering. Then one checks that
\begin{align*}
 & (\overline{s}_{i,i+1}\overline{x}_{i+1})\overline{x}_{i}  \\
=& (\overline{x}_{i}\overline{s}_{i,i+1}-\mathbf{k}_{\alpha_{i,i+1}}(-1+\overline{c}_{i+1}\overline{c}_{i}))\overline{x}_{i}  \\
=& \overline{x}_i\overline{s}_{i,i+1}\overline{x}_{i}-\mathbf{k}_{\alpha_{i,i+1}}(-1+\overline{c}_{i+1}\overline{c}_{i})\overline{x}_{i}  \\
=& \overline{x}_i\overline{x}_{i+1}\overline{s}_{i,i+1}+\mathbf{k}_{\alpha_{i,i+1}}\overline{x}_i(-1+\overline{c}_i\overline{c}_{i+1})-\mathbf{k}_{\alpha_{i,i+1}}(-1+\overline{c}_{i+1}\overline{c}_{i})\overline{x}_{i}  \\
=& \overline{x}_i\overline{x}_{i+1}\overline{s}_{i,i+1}
\end{align*}
and
\begin{align*}
 & \overline{s}_{i,i+1}(\overline{x}_{i+1}\overline{x}_{i})  \\
=& \overline{s}_{i,i+1}(\overline{x}_{i}\overline{x}_{i+1}+N_{B_n}\overline{c}_i\overline{c}_{i+1}) \\
=&  \overline{s}_{i,i+1}\overline{x}_{i}\overline{x}_{i+1}+N_{B_n}\overline{c}_{i+1}\overline{c}_{i}\overline{s}_{i,i+1} \\
=& \overline{x}_{i+1}\overline{s}_{i,i+1}\overline{x}_{i+1}+\mathbf{k}_{\alpha_{i,i+1}}(-1+\overline{c}_i\overline{c}_{i+1})x_{i+1} +N_{B_n}\overline{c}_{i+1}\overline{c}_{i}\overline{s}_{i,i+1} \\
=& \overline{x}_{i+1}\overline{x}_{i}\overline{s}_{i,i+1}-\mathbf{k}_{\alpha_{i,i+1}}\overline{x}_{i+1}(-1+\overline{c}_{i+1}\overline{c}_{i})+\mathbf{k}_{\alpha_{i,i+1}}(-1+\overline{c}_i\overline{c}_{i+1})\overline{x}_{i+1}+N_{B_n}\overline{c}_{i+1}\overline{c}_{i}\overline{s}_{i,i+1}  \\
=& \overline{x}_i\overline{x}_{i+1}\overline{s}_{i,i+1}
\end{align*}
Similarly,
\begin{align*}
 &(\overline{s}_n\overline{x}_n)\overline{x}_j  \\
=& (-\overline{x}_n\overline{s}_n-\sqrt{2}\mathbf{k}_{\alpha_n})\overline{x}_j \\
=& -\overline{x}_n\overline{s}_n\overline{x}_j -\sqrt{2}\mathbf{k}_{\alpha_n}\overline{x}_j \\
=& -\overline{x}_n\overline{x}_j\overline{s}_n-\sqrt{2}\mathbf{k}_{\alpha_n}\overline{x}_j \\
=& -\overline{x}_j \overline{x}_n \overline{s}_n -N_{B_n}c_jc_n\overline{s}_n -\sqrt{2}\mathbf{k}_{\alpha_n}\overline{x}_j
\end{align*}
and
\begin{align*}
 &\overline{s}_n(\overline{x}_n\overline{x}_j)  \\
=& \overline{s}_n\overline{x}_j\overline{x}_n +N_{B_n}\overline{s}_n \overline{c}_j\overline{c}_n \\
=& \overline{x}_j\overline{s}_n\overline{x}_n +N_{B_n}\overline{s}_n \overline{c}_j\overline{c}_n \\
=& -\overline{x}_j\overline{x}_n\overline{s}_n-\sqrt{2}\mathbf{k}_{\alpha_n}\overline{x}_j - N_{B_n}\overline{c}_j\overline{c}_n\overline{s}_n
\end{align*}
Similarly, for $i >j>k$,
\begin{align*}
 &  (\overline{x}_i\overline{x}_j)\overline{x}_k \\
=&  \overline{x}_j\overline{x}_i\overline{x}_k +N_{B_n}c_jc_i\overline{x}_k \\
=& \overline{x}_j \overline{x}_k \overline{x}_i +N_{B_n}\overline{x}_jc_kc_i+N_{B_n}c_jc_i\overline{x}_k \\
=& \overline{x}_k \overline{x}_j \overline{x}_i +N_{B_n}c_kc_j\overline{x}_i +N_{B_n}\overline{x}_jc_kc_i+N_{B_n}c_jc_i\overline{x}_k \\
=& \overline{x}_k\overline{x}_j \overline{x}_i+N_{B_n} c_kc_j\overline{x}_i+N_{B_n} c_kc_i\overline{x}_j +N_{B_n}c_jc_i\overline{x}_k
\end{align*}
The calculation for $\overline{x}_i(\overline{x}_j\overline{x}_k)$ is similar. Other minimal ambiguities can be checked similarly.

Let $\mathcal I$ be the two-sided ideal generated of $\widetilde{\mathbb{H}}$ by the relations of $W(B_n)$ (e.g. $\overline{s}^2-1$, $\overline{s}_{i,i+1}\overline{s}_{i+1,i+2}\overline{s}_{i,i+1}-\overline{s}_{i+1,i+2}\overline{s}_{i,i+1}\overline{s}_{i+1,i+2}$). Then $\widetilde{H}/\mathcal I \cong \mathbb{H}^{Cl}_{W(B_n)}$. Let $\mathcal P$ be the subalgebra of $\widetilde{\mathbb{H}}$ generated by by $x_i$ and $c_i$. It is straightforward to check that
\[(\overline{s}^2-1)\mathcal P=\mathcal P(\overline{s}^2-1),\]
\[(\overline{s}_{i,i+1}\overline{s}_{i+1,i+2}\overline{s}_{i,i+1}-\overline{s}_{i+1,i+2}\overline{s}_{i,i+1}\overline{s}_{i+1,i+2})\mathcal P=\mathcal P(\overline{s}_{i,i+1}\overline{s}_{i+1,i+2}\overline{s}_{i,i+1}-\overline{s}_{i+1,i+2}\overline{s}_{i,i+1}\overline{s}_{i+1,i+2}),\]
and other similar equations. Those equations can also be deduced from Lemma \ref{lem commut rel2} and its proof below.

\end{proof}

\begin{lemma}
For any root $\alpha >0$, $c_i\widetilde{s}_{\alpha}=-\widetilde{s}_{\alpha}c_i$.
\end{lemma}

\begin{lemma} \label{lem commut rel2}
\[ wy_iw^{-1}= y_{w(i)}+\sqrt{2} \sum_{\alpha>0, w^{-1}(\beta)<0, \langle \beta, w(e_i) \rangle \neq 0}\mathbf{k}_{\alpha} \widetilde{s}_{\alpha}. \]
In particular, for $\alpha>0$
\[ \widetilde{s}_{\alpha}y_i\widetilde{s}_{\alpha}^{-1}+y_{s_{\alpha}(i)}= -\sqrt{2} \sum_{\alpha>0, s_{\alpha}^{-1}(\beta)<0, \langle \beta, s_{\alpha}(e_i) \rangle \neq 0}\mathbf{k}_{\beta} \widetilde{s}_{\beta}. \]
\end{lemma}

\begin{proof}
The relation $s_nx_n+x_ns_n=-\sqrt{2}\mathbf{k}_{\alpha}$ implies $s_ny_n-y_ns_n=-\sqrt{2}\mathbf{k}_{\alpha}c_n$. The latter equation is also equivalent to $s_ny_ns_n^{-1}=y_n+\sqrt{2} \mathbf{k}_{\alpha}\widetilde{s}_n$. The remaining proof is just similar to the case of $A_{n-1}$ in the proof of Lemma \ref{lem comm form}.
\end{proof}
For $i>0$,  define $y_i=x_ic_i$.

\begin{align} \label{eqn yi' Bn}
   y_i'& =y_i+\frac{\sqrt{2}}{2}\sum_{\alpha >0, \langle \alpha, e_i\rangle \neq 0} \mathbf{k}_{\alpha}\widetilde{s}_{\alpha}.
\end{align}
We also efine $y_{-i}=y_i$ and $y_{-i}'=y_i'$.

 There is a natural permutation of $W(B_n)$ on the set $\left\{ \pm 1, \ldots, \pm n \right\}$.

\begin{lemma} \label{def hca Bn}
\begin{enumerate}
\item  For any $w\in W(B_n)$, $wy_i'w^{-1}=y_{w(i)}'$;
\item For $i \neq j$, $y_i'y_j'+y_j'y_i' \in \Seg(W(B_n))$.
\end{enumerate}
\end{lemma}

\begin{proof}
For (1), it suffices to check when $w=s_{\alpha}$ is a simple reflection. It is the direct consequence of the expression (\ref{eqn yi' Bn}) for $y_i'$, Lemma \ref{lem commut rel2}, and the fact that $s_{\alpha}\widetilde{s}_{\alpha}s_{\alpha}=-\widetilde{s}_{\alpha}$.
For (2), using the expression (\ref{eqn yi' Bn}), we have
\begin{eqnarray*}
& & y_i'y_j'+y_j'y_i'  \\
&=& y_iy_j+y_jy_i+\frac{\sqrt{2}}{2}\left[\sum_{\alpha>0, \langle \alpha, e_j\rangle \neq 0} \mathbf{k}_{\alpha}\left(y_i\widetilde{s}_{\alpha}+\widetilde{s}_{\alpha} y_i\right)+\sum_{\alpha>0, \langle \alpha, e_i\rangle \neq 0}\mathbf{k}_{\alpha} \left(y_j\widetilde{s}_{\alpha}+\widetilde{s}_{\alpha} y_j\right) \right] \\
& & \quad +\frac{1}{2}\sum_{\alpha, \beta>0, \langle \alpha, e_i\rangle \neq 0,  \langle \beta, e_i\rangle \neq 0} \mathbf{k}_{\alpha}\mathbf{k}_{\beta} (\widetilde{s}_{\alpha}\widetilde{s}_{\beta}+\widetilde{s}_{\beta}\widetilde{s}_{\alpha})
\end{eqnarray*}
Since $y_iy_j+y_jy_i=N_{B_n}$, we only need to consider and show the middle term is in $\Seg(W(B_n))$:
\begin{align*}
& \sum_{\alpha>0, \langle \alpha, e_j\rangle \neq 0} \left(y_i\widetilde{s}_{\alpha}+\widetilde{s}_{\alpha} y_i\right)+\sum_{\alpha>0, \langle \alpha, e_i\rangle \neq 0} \left(y_j\widetilde{s}_{\alpha}+\widetilde{s}_{\alpha} y_j\right)  \\
=&\left(\sum_{k \neq i}\mathbf{k}_{\alpha_{i,k}} \widetilde{s}_{i,k} y_j+y_j\sum_{k \neq i}\mathbf{k}_{\alpha_{i,k}}\widetilde{s}_{i,k}+y_i \sum_{l \neq j}\mathbf{k}_{\alpha_{l,j}} \widetilde{s}_{l,j}+\sum_{l \neq j}\mathbf{k}_{\alpha_{l,j}} \widetilde{s}_{l,j}y_i\right) \\
 & \quad \quad +\left(\sum_{k>0, k \neq i} \widetilde{s}_{i,-k} y_j+y_j\sum_{k>0, k \neq i}\widetilde{s}_{i,-k}+y_i \sum_{l>0 ,l \neq j} \widetilde{s}_{j,-l}+\sum_{l>0, l \neq j} \widetilde{s}_{j,-l}y_i\right) \\
 & \quad \quad +\mathbf{k}_{\alpha_i}(\widetilde{s}_{i}y_j+y_j\widetilde{s}_{i})+\mathbf{k}_{\alpha_j}(\widetilde{s}_j y_i+y_i\widetilde{s}_i)
\end{align*}
which is in $\mathrm{Seg}(W(B_n))$ by Lemma \ref{lem commut rel2}.

\end{proof}

\begin{proposition} \label{prop star type B}
The superalgebra $\mathbb{H}^{Cl}_{W(B_n)}$ satisfies the property (*).
\end{proposition}

\begin{proof}
Let $x_i'=-y_i'c_i$. We set $W$ in Definition \ref{ss HW} to be $W(B_n)$ and $a_i$ to be $x_i'$. With Lemma \ref{def hca Bn}, one can verify relations (\ref{rel star 0}) to (\ref{rel star 4}) in Definition \ref{ss HW} (also see more detail for type $A_{n-1}$ in Section \ref{ss comm rel}). By Proposition \ref{prop PBW Bn}, $\mathbb{H}^{Cl}_{W(B_n)}$ has a PBW type basis. These show the proposition.

\end{proof}

\subsection{Type $D_n$}

\begin{notation}
Let $W(D_n)$ be the Weyl group of type $D_n$.  Let the set $R(D_n)$ of roots for type $D_n$ be
\[  R(D_n) = \left\{ \pm e_i \pm e_j : 1 \leq i < j \leq n \right\}  \subset R(B_n).
\]
Let $R^+(D_n)=R(D_n) \cap R^+(B_n)$ be a fixed set of positive roots. We shall again write $\alpha>0$ for $\alpha \in R^+(D_n)$ and $\alpha <0$ for $-\alpha \in R(D_n)$. The set simple roots is given by
\[ \Delta = \left\{ e_i-e_{i+1} : i =1, \ldots, n-1 \right\} \cup \left\{ e_{n-1}+e_n \right\} .\]
Since there is only one $W$-orbit for $R(D_n)$, we simply write $\mathbf{k}$ for $\mathbf{k}_{\alpha}$ for any $\alpha \in R(D_n)$.

We shall regard $W(D_n)$ as the subgroup of $W(B_n)$ generated by elements $s_{i,j}$ and $s_{i,-j}$ for $i,j>0$. We shall also keep using notations in Notation \ref{notation Bn}.
\end{notation}

\begin{definition}  \label{def hca Dn}
Let $N_{D_n} \in \mathbb{C}$. Let $\mathbf{k}^B:R(B_n) \rightarrow \mathbb{C}$ such that $\mathbf{k}^B|_{R(D_n)} =\mathbf{k}$ and $\mathbf{k}^B_{\alpha}=0$ for any short root $\alpha$ in $R(B_n)$. Let $\mathbb{H}^{Cl}_{W(D_n)}=\mathbb{H}^{Cl}_{W(D_n)}(\mathbf{k}, N_{D_n})$ be the super subalgebra of $\mathbb{H}^{Cl}_{W(B_n)}(\mathbf{k}^B,N_{D_n})$ generated by the elements $w \in W(D_n) \subset W(B_n)$, $\left\{ x_i \right\}_{i=1}^n$ and $\left\{ c_i \right\}_{i=1}^n$.

\end{definition}

\begin{remark} \label{rmk Dn relation}
We can explicitly write down the commutation formula from the algebra structure of $\mathbb{H}_{W(B_n)}^{Cl}$. For example,
\begin{align*}
  & s_{n-1,-n}x_{n-1}+x_ns_{n-1,-n} \\
=& s_ns_{n-1,n}s_nx_{n-1}+x_ns_ns_{n-1,n}s_n \\
=& -s_ns_{n-1,n}x_ns_n+s_nx_{n-1}s_{n-1,n}s_n \\
=& s_n(-s_{n-1,n}x_n+x_{n-1}s_{n-1,n})s_n \\
=& s_n(\mathbf{k}(-1+c_{n}c_{n-1}))s_n \\
=& \mathbf{k}(-1+c_{n-1}c_n)
\end{align*}

This agrees with a relation in \cite[Definition 3.6]{WK}. When $N_{D_n}=0$, $\mathbb{H}^{Cl}_{W(D_n)}(\mathbf{k}, 0)$ is isomorphic to the degenerate affine Hecke-Clifford algebra of type $D_n$ defined in \cite[Definition 3.6]{WK}.  (We remark that in \cite{WK}, their convention for $c_i$ satisfying $c_i^2=1$ rather than $c_i^2=-1$.)

\end{remark}


We again define
\begin{align} \label{eqn yi' Dn}   y_i'=y_i+\frac{\sqrt{2}}{2}\mathbf{k}\sum_{\alpha>0, \langle \alpha, e_i\rangle \neq 0} \widetilde{s}_{\alpha}=y_i+\frac{\sqrt{2}}{2}\mathbf{k} \sum_{j \neq i} \widetilde{s}_{ij}+\frac{\sqrt{2}}{2}\mathbf{k}\sum_{j \neq i}\widetilde{s}_{i,-j}.
\end{align}
Again, for notational convenience, set $y_{-i}'=y_i'$.


\begin{lemma} \label{lem comm rel Dn}
\begin{enumerate}
\item $c_iy_j'=-y_j'c_i$ for any $i,j$;
\item  $s_{\alpha}y_i's_{\alpha}^{-1}=y_{s_{\alpha}(i)}'$;
\item For $i \neq j$, $y_i'y_j'+y_j'y_i' \in \Seg(W(D_n))$.
\end{enumerate}
\end{lemma}

\begin{proof}
Note that $y_i'$ is defined as the one in (\ref{eqn yi' Bn}) for type $B_n$ in Section \ref{ss Dirac type Bn} since we have $\mathbf{k}^B_{\alpha}=0$ for any short root $\alpha \in R(B_n)$. Then the results can be established by Lemma \ref{lem commut rel2} and investigating the proof of Lemma \ref{def hca Bn}.
\end{proof}

\begin{proposition} \label{prop star type D}
The algebra $\mathbb{H}^{Cl}_{W(D_n)}$ satisfies the property (*) in Definition \ref{def hw property star}.
\end{proposition}

\begin{proof}
This follows from $\mathbb{H}^{Cl}_{W(D_n)}$ forms a super subalgebra of $\mathbb{H}^{Cl}_{W(B_n)}(\mathbf{k}^B,N_{D_n})$ and Remark \ref{rmk Dn relation}.
\end{proof}

\subsection{Dirac element $D$}
Let $\mathbb{H}=\mathbb{H}^{Cl}_{W(A_{n-1})}$, $\mathbb{H}^{Cl}_{W(B_n)}$ or $\mathbb{H}^{Cl}_{W(D_n)}$. Using (\ref{eqn Dirac form}), the Dirac element $D$ for $\mathbb{H}$ is defined as
\begin{eqnarray} \label{eqn D hcl}    D= \sum_{i=1}^n x_i'c_i  .
\end{eqnarray}
Using the expressions in Section \ref{ss comm rel}, the explicit form of the Dirac element $D$ is as:
\begin{enumerate}
\item Type $A_{n-1}$ and $D_n$
\[  D= \sum_{i=1}^nx_ic_i +\sqrt{2} \sum_{\alpha >0}\mathbf{k}_{\alpha} s_{\alpha}c_{\alpha}=\sum_{i=1}^ny_i +\sqrt{2} \sum_{\alpha >0}\mathbf{k}_{\alpha} \widetilde{s}_{\alpha} . \]
\item Type $B_n$
\[ D=\sum_{i=1}^ny_i +\sqrt{2} \sum_{\alpha >0,\ \alpha \ long}\mathbf{k}_{\alpha} \widetilde{s}_{\alpha}+\frac{\sqrt{2}}{2} \sum_{\alpha >0,\ \alpha \ short}\mathbf{k}_{\alpha} \widetilde{s}_{\alpha} \]
\end{enumerate}

In type $A_{n-1}$ and $D_n$, we consider all the roots are long.

\begin{lemma} \label{lem square of sum refl}
\[  \left(\sum_{\alpha >0, \alpha \ long} \widetilde{s}_{\alpha}\right)^2 = \sum_{\substack{\alpha>0, \beta>0,s_{\alpha}(\beta)<0\\ \alpha, \beta \ long}} \widetilde{s}_{\alpha}\widetilde{s}_{\beta}  .\]
The above equality is also true if we replace all the long roots by short roots. Similarly, we also have
\[  \left(\sum_{\alpha >0, \alpha \ long} \widetilde{s}_{\alpha}\right) \left(\sum_{\alpha >0, \alpha \ short} \widetilde{s}_{\alpha}\right)+\left(\sum_{\alpha >0, \alpha \ short} \widetilde{s}_{\alpha}\right) \left(\sum_{\alpha >0, \alpha \ long} \widetilde{s}_{\alpha}\right) = \sum_{\substack{\alpha>0, \beta>0,s_{\alpha}(\beta)<0}} \widetilde{s}_{\alpha}\widetilde{s}_{\beta}  ,\]
where $\alpha$ and $\beta$ run for all pairs of root with distinct length.
\end{lemma}

\begin{proof}
We only prove for the first case, that is the case of long roots only. It suffices to show that
\[   \sum_{\substack{\alpha>0, \beta>0,s_{\alpha}(\beta)>0 \\ \alpha, \beta \ long}} \widetilde{s}_{\alpha}\widetilde{s}_{\beta} =0 . \]
Set $\widetilde{R}= \left\{ (\alpha, \beta) \in R^+ \times R^+: s_{\alpha}(\beta) >0, \mbox{$\alpha$ and $\beta$ are long} \right\}$. Note that for any $(\alpha, \beta) \in \widetilde{R}$, either $s_{\beta}(\alpha)>0$ or $s_{s_{\alpha}(\beta)}(\alpha)>0$. We define a map $\iota: \widetilde{R} \rightarrow \widetilde{R}$ such that
\[ \iota(\alpha, \beta) = \left\{ \begin{array}{c c}
                                       (\beta, s_{\beta}(\alpha)) & \mbox{ if $s_{\beta}(\alpha)>0$ } \\
															         (s_{\alpha}(\beta),\alpha) & \mbox{ if $s_{s_{\alpha}(\beta)}(\alpha)>0$ }
                  \end{array} \right.
									\]
It is not hard to verify $\iota$ is well-defined and is an involution. For $\iota(\alpha, \beta)=(\alpha', \beta')$, one can also check that $\widetilde{s}_{\alpha}\widetilde{s}_{\beta}+\widetilde{s}_{\alpha'}\widetilde{s}_{\beta'}=0$. Thus each term $\widetilde{s}_{\alpha}\widetilde{s}_{\beta}$ in the expression $\sum_{\alpha>0, \beta>0,s_{\alpha}(\beta)>0} \widetilde{s}_{\alpha}\widetilde{s}_{\beta}$ can be paired with another one and gets canceled. This proves the expression is zero.
\end{proof}

By Proposition \ref{prop isom hcl}, Proposition \ref{prop star type B} and Proposition \ref{prop star type D}, $\mathbb{H}$ satisfies the property (*) and hence we can define $\mathrm{Seg}(W)$ to be a subalgebra of $\mathbb{H}$ according to Definition \ref{def hw property star}.

We compute the square of the Dirac element $D$. This is an analogue of \cite[Theorem 2.11]{BCT}.
\begin{theorem} \label{thm d sq}
Let $\mathbb{H}=\mathbb{H}^{Cl}_{W(A_{n-1})}$, $\mathbb{H}^{Cl}_{W(B_n)}$ or $\mathbb{H}^{Cl}_{W(D_n)}$. Then
\[   D^2 = \Omega_{\mathbb{H}} - \Omega_{\Seg(W)} , \]
where
\[  \Omega_{\mathbb{H}} = \sum_{i=1}^n x_i^2 , \]
\[  \Omega_{\Seg(W)} =  \frac{1}{2}\sum_{\alpha>0, \beta>0, s_{\alpha}(\beta)<0} |\langle \alpha , \alpha \rangle||\langle \beta, \beta \rangle|  \mathbf{k}_{\alpha}\mathbf{k}_{\beta} \widetilde{s}_{\alpha}\widetilde{s}_{\beta} .\]
Moreover, $\mathbb{H}$ satisfies the property (**).
\end{theorem}

\begin{proof}
We only do for types $A_{n-1}$ and $B_n$ and the case for type $D_n$ follows from type $B_n$.

By Lemma \ref{lem comm form} and Lemma \ref{lem commut rel2}, for any $\alpha >0$,
\begin{eqnarray} \label{eqn d2 1}
 \mathbf{k}_{\alpha} \left(\sum_{i=1}^n y_i \widetilde{s}_{\alpha}+\widetilde{s}_{\alpha}\sum_{i=1}^ny_i\right) &=& - \sqrt{2}\mathbf{k}_{\alpha} \sum_{\beta>0, s_{\alpha}(\beta)<0}\mathbf{k}_{\beta} |\langle \beta, \beta \rangle| \widetilde{s}_{\alpha}\widetilde{s}_{\beta}
\end{eqnarray}
Now, by (\ref{eqn d2 1}) and Lemma \ref{lem square of sum refl},
\begin{eqnarray*}
   D^2 & = & \left( \sum_{i=1}^n y_i + \frac{\sqrt{2}}{2}\sum_{\alpha>0}\mathbf{k}_{\alpha}|\langle \alpha, \alpha \rangle| \widetilde{s}_{\alpha} \right)^2 \\
       & = & \left( \sum_{i=1}^n y_i \right)^2 +\frac{\sqrt{2}}{2}\sum_{i=1}^n y_i \sum_{\alpha >0}\mathbf{k}_{\alpha}|\langle \alpha, \alpha \rangle|\widetilde{s}_{\alpha}+\frac{\sqrt{2}}{2}\sum_{\alpha >0}\mathbf{k}_{\alpha}|\langle \alpha, \alpha \rangle|\widetilde{s}_{\alpha} \sum_{i=1}^n y_i +\frac{1}{2}\left( \sum_{\alpha >0}\mathbf{k}_{\alpha} |\langle \alpha, \alpha \rangle| \widetilde{s}_{\alpha} \right)^2 \\
       & = & \sum_{i=1}^n x_i^2 -\frac{1}{2}\sum_{\alpha >0, \beta>0, s_{\alpha}(\beta)<0}\mathbf{k}_{\alpha}\mathbf{k}_{\beta} |\langle \alpha, \alpha \rangle||\langle \beta, \beta\rangle| \widetilde{s}_{\alpha}\widetilde{s}_{\beta}
\end{eqnarray*}

We can directly verify that $\Omega_{\mathbb{H}}$ is in the center of $\mathbb{H}$ and $\Omega_{\mathrm{Seg}(W)}$ is in the center of $\Seg(W)$. Hence, $\mathbb{H}$ has the property (**).
\end{proof}

We obtain the following Parthasarathy-Dirac-type inequality. Examples satisfying the hypothesis of Corollary \ref{cor par dirac inq} below will be considered in Section \ref{s spec} (see Proposition \ref{lem adj op}).
\begin{corollary} \label{cor par dirac inq}
Suppose an irreducible $\mathbb H$-module $(\pi, X)$ satisfies the property that $X$ admits a non-degenerate positive-definite Hermitian form such that the adjoint operator of $\pi(D)$ is $-\pi(D)$. For any irreducible $\Seg(W)$-module $(\sigma, U)$,
\[  \mathrm{Hom}_{\Seg(W)}(U, \Res_{\Seg(W)}^{\mathbb{H}}X) \neq 0 \]
only if
\[ \chi_{\pi}(\Omega_{\mathbb{H}}) \leq \chi_{\sigma}(\Omega_{\Seg(W)})  .\]

\end{corollary}

\begin{proof}
Let $U_X$ be an $U$-isotypical component of $X$ and let $u \in U$. The corollary follows from
\[  0 \leq \langle D.u, D.u \rangle=\langle u, -D^2. u \rangle = -(\chi_{\pi}(\Omega_{\mathbb{H}}) - \chi_{\sigma}(\Omega_{\Seg(W)}))\langle u, u \rangle .
\]
\end{proof}

The conclusion of this section is  a version of Theorem \ref{thm vog conj hw} in specific cases.

\begin{theorem} \label{thm Vogan conjecture HC}
Let $\mathbb{H}=\mathbb{H}^{Cl}_{W(A_{n-1})}$, $\mathbb{H}^{Cl}_{W(B_n)}$ or $\mathbb{H}^{Cl}_{W(D_n)}$. Let $(\pi, X)$ be an irreducible supermodule of $\mathbb{H}$ with the central character $\chi_{\pi}$ (Definition \ref{def central character}). Let $D$ be the Dirac element in $\mathbb{H}$ in (\ref{eqn D hcl}). Define the Dirac cohomology $H_D(X)$ as in Theorem \ref{thm vog conj hw}. Then $H_D(X)$ has a natural $\Seg(W)$-module structure. Suppose
 \[\Hom_{\Seg(W)}(U, H_D(X)) \neq 0,\]
for some $\Seg(W)$-module $(\sigma, U)$.  Then $\chi_{\pi}= \chi^{\sigma}$, where $\chi^{\sigma}$ is defined as in (\ref{eqn central character sigma})
\end{theorem}

\begin{proof}
This immediately follows from Theorem \ref{thm vog conj hw}, Proposition \ref{prop isom hcl} and Theorem \ref{thm d sq}.
\end{proof}

\section{Examples of non-vanishing Dirac cohomology} \label{s ex vanishing Dirac}

\subsection{Construction of some modules} \label{ss steinberg all}

In this section, we construct some modules for the degenerate affine Hecke-Clifford algebra of classical types. The underlying idea of the construction is to first consider a $\Seg(W)$-module and then try to extend the action to the entire degenerate affine Hecke-Clifford algebra. However, we may not expect this process always works and indeed, we can only do it for certain parameters.

In type $A_{n-1}$, we follow the construction in \cite[Section 4.1]{HKS}, which uses a Jucys-Murphy type element. For types $B_n$ (in certain parameter cases), we use a slightly different approach. \\

\noindent
{\bf Type $A_{n-1}$}
Let $\mathrm{Cl}_n$ be the subalgebra of $\mathbb{H}^{Cl}_{W(A_{n-1})}$ generated by all $c_i$. Define  $ \widetilde{\mathrm{St}}_{W(A_{n-1})}$ to be an $\mathbb{H}^{Cl}_{W(A_{n-1})}$-supermodule, which is identified with $\mathrm{Cl}_{n}$ as vector spaces and the action of $\mathbb{H}^{n}$ on $ \widetilde{\mathrm{St}}_{W(A_{n-1})}$  is  determined by the following:
\begin{eqnarray}
\label{eqn action st 1}         c_i. 1 &=& c_i , \\
\label{eqn action st 2}      s_{\alpha}. 1 &=& 1,
\end{eqnarray}
where $1$ is the identity in $\mathrm{Cl}_n$ and
\[    x_i. v= \mathbf{k}\left(\sum_{1 \leq j<i \leq n} s_{i,j}(1-c_ic_j)\right). v , \]
where $v$ is any vector in $\mathrm{Cl}_n$ and the actions of $s_{i,j}$ and $c_i,c_j$ are the ones defined in (\ref{eqn action st 1}) and (\ref{eqn action st 2}). The notation $ \widetilde{\mathrm{St}}_{W(A_{n-1})}$ stands for a Steinberg type module as it performs the role of Steinberg module in the degenerate affine Hecke algebra. It is straightforward to check the above actions define an $\mathbb{H}^{Cl}_{W(A_{n-1})}$-module by verifying the defining relations of $\mathbb{H}^{Cl}_{W(A_{n-1})}$. Some details can be found in \cite[Proposition 4.1.1]{HKS}. \\

\noindent
{\bf Type $B_n$ }
Let $\alpha$ be a long root in $R(B_n)$ and let $\beta$ be a short root in $R(B_n)$. Set $N_{B_n}=2(n-1)\mathbf{k}_{\alpha}^2+\sqrt{2}\mathbf{k}_{\alpha}\mathbf{k}_{\beta}$. Let $\mathrm{Cl}_n$ be the subalgebra of $\mathbb{H}^{Cl}_{W(B_n)}$ generated by the elements $c_i$, which is isomorphic to the Clifford algebra. Let $U(n)$ be an irreducible supermodule of $\mathrm{Cl}_n$.
The actions of $\mathbb{H}^{Cl}_{W(B_n)}$ on $U(n) \widetilde{\otimes} U(n)$ is determined by the following:
\begin{align}
\label{eqn bn st 1} x_{i}. (u\otimes v) & =-(-1)^{\mathrm{deg}(u)}\sqrt{-1}\left( \left(\mathbf{k}_{\alpha}(c_1+c_2+\ldots+c_{i-1}+(n-i)c_i)+ \frac{\sqrt{2}}{2}\mathbf{k}_{\beta}c_i \right). u\right) \otimes (c_{i}.v),
\\
  s_n. (u \otimes v) &=(-1)^{\mathrm{deg}(u)} \sqrt{-1}(c_n.u) \otimes (c_n.v) ,\\
\label{eqn bn st 4}   s_{i,j}. (u \otimes v) &= (-1)^{\mathrm{deg}(u)}\sqrt{-1}\left( \frac{c_i-c_j}{\sqrt{2}} .u\right) \otimes \left(\frac{c_i-c_j}{\sqrt{2}} .v\right) , \\
\label{eqn bn st 5}  c_i. (u \otimes v) &= (-1)^{\mathrm{deg}(u)}(u \otimes c_i.v) .
\end{align}

The above actions are indeed well-defined:

\begin{proposition}
For $N_{B_n}=2(n-1)\mathbf{k}_{\alpha}^2+\sqrt{2}\mathbf{k}_{\alpha}\mathbf{k}_{\beta}$, the actions (\ref{eqn bn st 1})-(\ref{eqn bn st 5}) above on $U(n) \widetilde{\otimes} U(n)$ define an $\mathbb{H}^{Cl}_{W(B_n)}(\mathbf{k},N_{B_n} )$-module.
\end{proposition}

\begin{proof}
The computation is straightforward for verifying the defining relations of $\mathbb{H}^{Cl}_{W(B_n)}$. For example,
\begin{align*}
&(s_{i,i+1}x_i-x_{i+1}s_{i,i+1}).(u \otimes v)      \\
=&-\frac{1}{2}\mathbf{k}_{\alpha} ((-(n-i)+ (n-i)c_ic_{i+1}).u ) \otimes ((-1 +c_ic_{i+1}).v) \\
 &  \quad \quad \quad +\frac{1}{2}\mathbf{k}_{\alpha}( ((n-i-2)- (n-i)c_ic_{i+1}    ).u)  \otimes (1-c_ic_{i+1}).v) \\
=& \mathbf{k}_{\alpha}  u \otimes ((-1+c_ic_{i+1}).v) \\
=& \mathbf{k}_{\alpha}(-1+c_ic_{i+1}).(u \otimes v) .
\end{align*}
Moreover, for $i < j$, note that
\begin{align*}
 & \left( \mathbf{k}_{\alpha}(c_1+c_2+\ldots+c_{i-1}+(n-i)c_i)+\frac{\sqrt{2}}{2}\mathbf{k}_{\beta}c_i \right) \left( \mathbf{k}_{\alpha}(c_1+c_2+\ldots+c_{j-1}+(n-j)c_j) +\frac{\sqrt{2}}{2}\mathbf{k}_{\beta}c_j \right) \\
& \quad \quad  + \left(\mathbf{k}_{\alpha}(c_1+c_2+\ldots+c_{j-1}+(n-j)c_j)+\frac{\sqrt{2}}{2}\mathbf{k}_{\beta}c_j \right)  \left( \mathbf{k}_{\alpha}(c_1+c_2+\ldots+c_{i-1}+(n-i)c_i) +\frac{\sqrt{2}}{2}\mathbf{k}_{\beta}c_i\right) \\
=& -2(i-1)\mathbf{k}_{\alpha}^2+ 2\mathbf{k}_{\alpha}(-(n-i)\mathbf{k}_{\alpha}-\frac{\sqrt{2}}{2}\mathbf{k}_{\beta}) \\
=& -2(n-1)\mathbf{k}_{\alpha}^2-\sqrt{2}\mathbf{k}_{\alpha}\mathbf{k}_{\beta}
\end{align*}
and hence $x_ix_j-x_jx_i=(2(n-1)\mathbf{k}_{\alpha}^2+\sqrt{2}\mathbf{k}_{\alpha}\mathbf{k}_{\beta})c_jc_i$.
Other relations can be verified similarly (and more easily).
\end{proof}

Denote the above $\mathbb{H}_{W(B_n)}^{Cl}$-module by $\widetilde{\mathrm{St}}_{B_n}$. \\

\noindent
{\bf Type $D_n$}  Set $N_{D_n} =2(n-1)\mathbf{k}_{\alpha}^2$. Recall that $\mathbb{H}_{W(D_n)}^{Cl}$ is a subalgebra of $\mathbb{H}_{W(B_n)}^{Cl}(\mathbf{k}^B, N_{D_n} )$ (see $\mathbf{k}^B$ in Definition \ref{def hca Dn}). By checking the parameter function, we have an $\mathbb{H}_{W(B_n)}^{Cl}(\mathbf{k}^B, N_{D_n} )$-module $\widetilde{\mathrm{St}}_{B_n}$ defined above. Denote by $\widetilde{\mathrm{St}}_{D_n}$ the restriction of $\widetilde{\mathrm{St}}_{B_n}$ to an $\mathbb{H}_{W(D_n)}^{Cl}$-module.

\subsection{Dirac cohomology}
We keep using the notation in Section \ref{s ex vanishing Dirac}.

\begin{proposition} \label{prop vanishing st}
Set $N_{B_n}=2(n-1)\mathbf{k}_{\alpha}^2+\sqrt{2}\mathbf{k}_{\alpha}\mathbf{k}_{\beta}$ (with the notations in Section \ref{ss steinberg all}) and set $N_{D_n}=2(n-1)\mathbf{k}^2$. Let $\mathbb{H}=\mathbb{H}^{Cl}_{A_{n-1}}, \mathbb{H}^{Cl}_{B_n}(\mathbf{k}, N_{B_n})$ or $\mathbb{H}^{Cl}_{D_n}(\mathbf{k}, N_{D_n})$. Let $X=\widetilde{\mathrm{St}}_{A_{n-1}}, \widetilde{\mathrm{St}}_{B_n}$ or $\widetilde{\mathrm{St}}_{D_n}$ be an $\mathbb{H}$-module defined in Section \ref{ss steinberg all}. The Dirac operator $D$ acts identically zero on $X$. In particular, $H_D(X) \neq 0$.

\end{proposition}

\begin{proof}
Type $A_{n-1}$: For $v \in \widetilde{\mathrm{St}}_{A_{n-1}}$,
\begin{eqnarray*}
\pi(D)v &=&   \sum_{1 \leq j<i \leq n} s_{ij}(1-c_ic_j)c_i.v+\sqrt{2}\mathbf{k}\sum_{\alpha \in R^+}\widetilde{s}_{\alpha}.v \\
        &=& \sum_{1\leq j<i \leq n} \mathbf{k}s_{ij}(c_i-c_j).v+\sqrt{2}\mathbf{k}\sum_{\alpha \in R^+}\widetilde{s}_{\alpha}.v \\
        &=&\left(- \sqrt{2} \sum_{1<j<i<n} \mathbf{k}\widetilde{s}_{ji} +\sqrt{2}\mathbf{k}\sum_{\alpha \in R^+}\widetilde{s}_{\alpha}\right).v\\
        &=&0 \\
\end{eqnarray*}

\noindent
Type $B_n$: Recall that $\widetilde{\mathrm{St}}_{B_n}$ is isomorphic to $U \widetilde{\otimes} U$ as vector spaces in the notation of Section \ref{ss steinberg all}. For $u \otimes v \in U \widetilde{\otimes} U$,
\begin{align*}
 & (-1)^{\mathrm{deg}(u)}\sqrt{-1}\pi(D)(u \otimes v) \\
=& \sum_{i=1}^n\left( \left( \mathbf{k}_{\alpha}(c_1+c_2+\ldots+c_{i-1}+(n-i)c_i) +\frac{\sqrt{2}}{2}\mathbf{k}_{\beta}c_i\right). u\right) \otimes v \\
 & \quad -\sqrt{2}\mathbf{k}_{\alpha}\sum_{1\leq j<i \leq n} \left(\frac{c_i-c_j}{\sqrt{2}}.u \right)\otimes v-\sqrt{2}\mathbf{k}_{\alpha}\sum_{1\leq j<i \leq n} \left(\frac{c_i+c_j}{\sqrt{2}}.u\right) \otimes v -\frac{\sqrt{2}}{2}\mathbf{k}_{\beta} \sum_{i=1}^n (c_i.u) \otimes v\\
        =& \sum_{i=1}^n\left( \left( 2(n-i)c_i +\frac{\sqrt{2}}{2}\mathbf{k}_{\beta}c_i\right). u\right) \otimes v -\sqrt{2}\mathbf{k}_{\alpha}\sum_{i=1}^n (2(n-i)c_i.u) \otimes v -\frac{\sqrt{2}}{2}\mathbf{k}_{\beta} \sum_{i=1}^n (c_i.u) \otimes v\\
			=& 0 \\
\end{align*}

\noindent
Type $D_n$: Recall that $\mathbb{H}_{W(D_n)}^{Cl}$ is a subalgebra of $\mathbb{H}_{W(B_n)}^{Cl}(\mathbf{k}^B, N_{D_n} )$ (see the notation of $\mathbf{k}^B$ in Definition \ref{def hca Dn}). The Dirac operator for $\mathbb{H}_{W(D_n)}^{Cl}$ is the same as the Dirac operator for $\mathbb{H}_{W(B_n)}^{Cl}(\mathbf{k}^B, N_{D_n} )$. Then the vanishing result follows from the result for type $B_n$, which has just been proven.
\end{proof}

\section{Sergeev algebra} \label{s seg alg}

The main purpose of this section is to review several results about Sergeev algebra, which will be useful for computing the Dirac cohomology of some modules for  $\mathbb{H}^{Cl}_{W(A_{n-1})}$ in next section. Some results can also be formulated to other types and one may refer to \cite[Section 2]{WK}. Starting from this section, we consider type $A_{n-1}$ only and we shall usually  use the notation $S_n$ for $W(A_{n-1})$ (where $S_n$ represents the symmetric group). Write $R$ for $R(A_{n-1})$ and $R^+$ for $R^+(A_{n-1})$. Recall that $\Delta$ is the set of simple roots in $R$.

\subsection{The superalgebra $\mathbb{C}[\widetilde{S}_n]^-$} \label{ss sym}

Let $\widetilde{S}_n$ be the group generated by the elements $\psi, \widetilde{t}_{1,2}, \ldots, \widetilde{t}_{n-1,n}$ subject to the following relations:
\[    (\widetilde{t}_{i,i+1})^2 = 1 \]
\[  (\widetilde{t}_{i,i+1}\widetilde{t}_{i+1,i+2})^3=1  \quad \mbox{ for $i=1, \ldots, n-1$ }\]
\[ \widetilde{t}_{i,i+1} \widetilde{t}_{j,j+1}=\psi\widetilde{t}_{j,j+1}\widetilde{t}_{i,i+1} \quad \mbox{ for $|i-j|>1$ } ,\]
\[ \psi \widetilde{t}_{i,i+1}=\widetilde{t}_{i,i+1}\psi \quad \mbox{ for $i=1,\ldots, n-1$ }, \]
\[ \psi^2=1 . \]
Then $\widetilde{S}_n$ is a double cover of $S_n$ via the map determined by sending $\widetilde{t}_{\alpha_i}$ to the transposition between $i$ and $i+1$, and $\psi \mapsto 1$. We also sometimes write $\widetilde{t}_{\alpha_{i,i+1}}$ for $t_{i,i+1}$ if we want to refer to the simple root $\alpha_{i,i+1}$. Denote by $\mathbb{C}[\widetilde{S}_n]$ the group algebra of $\widetilde{S}_n$ with a basis labeled as $\left\{ e_{\widetilde{w}} : \widetilde{w} \in \widetilde{S}_n \right\}$. Define $\mathbb{C}[\widetilde{S}_n]^-:= \mathbb{C}[\widetilde{S}_n]/\langle e_{\psi} +1 \rangle$. We shall simply write $\widetilde{w}$ for the image of $e_{\widetilde{w}}$ in $\mathbb{C}[\widetilde{S}_n]^-$. There is a superalgebra structure on $\mathbb{C}[\widetilde{S}_n]^-$ with $\deg(\widetilde{t}_{\alpha})=1$ for all  $\alpha \in \Delta$.



\begin{lemma} \label{lem spin repns const}
Given a $S_n$-representation $U$ and a $\mathbb{C}[\widetilde{S}_n]^-$-module $U'$, there exists a natural $\mathbb{C}[\widetilde{S}_n]^-$-module structure on $U \otimes U'$ characterized by
\[    \widetilde{t}_{\alpha}. (u \otimes u') = (s_{\alpha}.u) \otimes (\widetilde{t}_{\alpha}. u') ,\]
where $\alpha \in \Delta$, $u \in U$ and $u' \in U'$.
\end{lemma}

Define an equivalence relation on $\Irr(\mathbb{C}[\widetilde{S}_n]^-)$: $U \sim_{\sgn} U'$ if and only if $U=U'$ or $U=\sgn \otimes U'$ as $\mathbb{C}[\widetilde{S}_n]^-$-modules, where $\sgn$ is the sign representation of $S_n$ and the $\mathbb{C}[\widetilde{S}_n]^-$-module structure of $\sgn \otimes U'$ is defined in Lemma \ref{lem spin repns const}.
\begin{proposition} \label{prop cri irr Sn}
There is a natural bijection
\[  \Irr_{\mathrm{sup}}(\mathbb{C}[\widetilde{S}_n]^-)/\sim_{\Pi}\ \longleftrightarrow \ \Irr(\mathbb{C}[\widetilde{S}_n]^-)/\sim_{\sgn} .\]
\end{proposition}

\begin{proof}
It suffices to see that the equivalence relation $\sim$ in Proposition \ref{prop bij sup and ord} is the same as $\sim_{\sgn}$. This follows from $\deg(\widetilde{t}_{\alpha})=1$ for all  $\alpha \in \Delta$ and definitions.

\end{proof}

\subsection{Sergeev algebra} \label{ss seg alg}

\begin{definition} \label{def seg alg}
Recall that $\mathbb{H}^{Cl}_{W(A_{n-1})}$ is defined in Definition \ref{def hca}. The Sergeev algebra, denoted $\Seg_n$, is the subalgebra of $\mathbb{H}^{Cl}_{W(A_{n-1})}$ generated by the elements $w \in W(A_{n-1})=S_n$ and $c_i$ ($i=1,\ldots, n$). In other words, since $\mathbb{H}^{Cl}_{W(A_{n-1})}$ satisfies the property (*),  $\mathrm{Seg}_n$ is the same as $\Seg(W_{A_{n-1}})$ in Definition \ref{def hw property star}. We shall use notations in Section \ref{ss comm rel} (e.g. $s_{\alpha}$, $c_{\alpha}$, $\widetilde{s}_{\alpha}$).

Let $\mathrm{Cl}_n$ be the super subalgebra of $\Seg_n$ generated by $c_i$ ($i=1,\ldots, n$). There exists a unique, up to applying the functor $\Pi$, irreducible supermodule of $\mathrm{Cl}_n$. Let $U(n)$ be a fixed choice of an irreducible supermodule of $\mathrm{Cl}_n$. The dimension of $U(n)$ is $2^{n/2}$ for $n$ even and $2^{(n+1)/2}$ for $n$ odd.

\end{definition}

The relation between subalgebras $\Seg_n$ and $\mathbb{C}[\widetilde{S}_n]^-$ is the following.

\begin{lemma} \label{lem seg iso} \cite[Lemma 13.2.3]{Kl}
 $\Seg_n$ is isomorphic to $\mathbb{C}[\widetilde{S}_n]^- \widetilde{\otimes} \mathrm{Cl}_n$ as superalgebras.
\end{lemma}

\begin{proof}
Define a map:
\[ s_{\alpha} \mapsto \widetilde{t}_{\alpha} \otimes c_{\alpha} \quad  (\alpha \in \Delta), \quad c_i  \mapsto 1\otimes c_i \quad (i=1,\ldots, n).\]
One can verify that the map is an isomorphism.
\end{proof}

For any $\alpha \in R^+$, define $\widetilde{t}_{\alpha} \in \mathbb{C}[S_n]^-$ such that $s_{\alpha}$ maps to $\widetilde{t}_{\alpha}\otimes c_{\alpha}$ under the map in the proof of Lemma \ref{lem seg iso}.

Here is an analogue of Lemma \ref{lem spin repns const}.

\begin{lemma} \label{lem seg mod const}
Given a $S_n$-representation $U$ and a $\Seg_n$-module $U'$, there exists a natural $\Seg_n$-module structure on $U \otimes U'$ characterized by
\[    s_{\alpha}. (u \otimes u') = (s_{\alpha}.u) \otimes (s_{\alpha}. u') ,\]
and
\[     c_i. (u \otimes u') =u \otimes (c_i.u') ,\]
where $\alpha \in \Delta$, $i=1,\ldots,n$, $u \in U$ and $u' \in U'$.
\end{lemma}

\subsection{Relation between supermodules of $\mathbb{C}[\widetilde{S}_n]^- $ and $\Seg_n$} \label{ss mod relations}


Recall from \cite{BK} (our formulation here is a bit different) a natural functor $F$ :
\[  F: \mathrm{Mod}_{\mathrm{sup}}(\mathbb{C}[\widetilde{S}_n]^-) \rightarrow \mathrm{Mod}_{\mathrm{sup}}(\Seg_n) , \]
\[      X \mapsto   X \otimes U(n) .\]
The $\Seg_n$-supermodule structure of $X \otimes U(n)$ is characterized by
\[  s_{\alpha}. (x \otimes u)=-(-1)^{\deg(x)}(\widetilde{t}_{\alpha}.x) \otimes (c_{\alpha}.u) \quad (\alpha \in \Delta),
\]
\[   c_i. (x \otimes u)=(-1)^{\deg(x)}x \otimes (c_i. u) \quad (i=1,\ldots, n).
\]
It is straightforward to check the above equations define a $\Seg_n$-module. Next, define
\[  G: \mathrm{Mod}_{\mathrm{sup}}(\Seg_n) \rightarrow \mathrm{Mod}_{\mathrm{sup}}(\mathbb{C}[\widetilde{S}_n]^-) , \]
\[     Y \mapsto \Hom_{\mathrm{Cl}_n}(U(n), Y) . \]
The $\mathbb{C}[\widetilde{S}_n]^-$-module structure is given by for $\theta \in \Hom_{\mathrm{Cl}_n}(U(n), Y)$,
\[     (\widetilde{t}_{\alpha}. \theta)(u)= (s_{\alpha}c_{\alpha}).\theta(u)  \quad (\alpha \in \Delta). \]

\begin{proposition} \cite[Theorem 3.4]{BK} \label{prop adj functor}
The functors $F$ and $G$ form an adjoint pair i.e. there is a natural isomorphism
\[  \Hom_{\Seg_n}(F(U), U')= \Hom_{\mathbb{C}[\widetilde{S}_n]^-}(U, G(U')) .\]
Furthermore if $n$ is even, $G \circ F=\Id$ and $F \circ G=\Id$. If $n$ is odd,
$G \circ F = \Id \oplus \Pi$ and $F \circ G=\Id \oplus \Pi$, where $\Pi$ is defined in Section \ref{ss prelim superalg}.
\end{proposition}

Let $U_{\mathrm{Cl}_n}$ be a $\Seg_n$-module defined by
\[  U_{\mathrm{Cl}_n} = \Ind^{\Seg_n}_{\mathbb{C}[S_n]} \triv =\Seg_n \otimes_{\mathbb{C}[S_n]} \triv,\]
where $\mathbb{C}[S_n]$ is regarded as the subalgebra of $\Seg_n$ generated by the elements $f_{s_{\alpha}}$ for all $\alpha \in \Delta$ and $\triv$ is the trivial representation of $\mathbb{C}[S_n]$. In particular, $\dim_{\mathbb{C}} U_{\mathrm{Cl}_n}= 2^n$.

We define a corresponding $\mathbb{C}[\widetilde{S}_n]^-$-module $U_{\mathrm{spin}}$ as follows. If $n$ is even, define $U_{\mathrm{spin}}=G(U_{\mathrm{Cl}_n})$. If $n$ is odd, by \cite[Proposition 13.2.2]{Kl} and \cite[Theorem 22.2.1]{Kl}, $G(U_{\mathrm{Cl}_n})=M \oplus \Pi(M)$ for some irreducible $\mathbb{C}[\widetilde{S}_n]^-$-module $M$. Then define $U_{\mathrm{spin}}=M$.

An immediate consequence of Proposition \ref{prop adj functor} is given below.


\begin{lemma} \label{ex cl spin}
$F(U_{\mathrm{spin}})=U_{\mathrm{Cl}_n}$.
\end{lemma}

\section{Spectrum of the Dirac operator for type $A_{n-1}$} \label{s spec}

We have seen the action of the Dirac operator on certain modules. In this section, we will go further for type $A_{n-1}$ and compute the action of $D$ on some interesting $\mathbb{H}^{Cl}_{W(A_{n-1})}$-modules. We shall see Theorem \ref{thm Vogan conjecture HC} for $\mathbb{H}^{Cl}_{W(A_{n-1})}$ has interesting consequences. We shall write $\mathbb{H}^{Cl}_n$ for $\mathbb{H}^{Cl}_{W(A_{n-1})}$ for simplicity. We keep using the notations in Section \ref{ss comm rel} and Section \ref{s seg alg}.
\subsection{Further notation for the root system of type $A_{n-1}$}

A partition of $n$ is a sequence of positive integers $(n_1, \ldots, n_r)$ such that $n_1 \geq n_2\geq \ldots \geq n_r$ and $n_1+\ldots+n_r=n$. For a partition $\lambda=(n_1, \ldots, n_r)$ of $n$, let $I_{\lambda}=\left\{1,\ldots, n \right\}\setminus \left\{n_1,n_1+n_2,\ldots, n_1+\ldots+n_r \right\}$ and let
\[ \Delta_{\lambda} = \left\{  e_i-e_{i+1} : i \in I_{\lambda} \right\} . \]
Let $V_{\lambda}$ be the real span of $\Delta_{\lambda}$ in $\mathbb{R}^n$ and let $R^+_{\lambda}=V_{\lambda} \cap R^+$.

\subsection{Central characters for $\mathbb{H}^{Cl}_n$} \label{ss def hecke clifford}

The center of $\mathbb{H}^{Cl}_n$ plays a role in the following computations.

\begin{proposition} \cite[Theorem 14.3.1]{Kl} \label{prop center HC}
The center  $Z(\mathbb{H}^{Cl}_n)$ of $\mathbb{H}^{Cl}_n$ is the set of all symmetric polynomials in $\mathbb{C}[x_1^2,x_2^2,\ldots, x_n^2]$. In particular, any element in  $Z(\mathbb{H}^{Cl}_n)$ is of even degree.
\end{proposition}

\begin{definition} \label{def cc}
Recall that the central character $\chi_{\pi}: Z(\mathbb{H}^{Cl}_n)_0 \rightarrow \mathbb{C}$ of an irreducible supermodule $(\pi,X)$ is defined in Definition \ref{def central character}. By Proposition \ref{prop center HC}, we can also write $\chi_{\pi}: Z(\mathbb{H}^{Cl}_n) \rightarrow \mathbb{C}$.

For an element $\gamma=(a_1,\ldots, a_n) \in \mathbb{C}^n$, define $\chi_{\gamma}': \mathbb{C}[x_1^2,\ldots, x_n^2] \rightarrow \mathbb{C}$ such that $\chi_{\gamma}'(x_i^2)=a_i$. Define $\chi_{\gamma}$ to be the restriction of $\chi_{\gamma}'$ to $Z(\mathbb{H}^{Cl}_n)$. For the central character $\chi_{\pi}$ of $X$, there exists a unique $\gamma \in \mathbb{C}^n$, up to permutations of coordinates, such that $\chi_{\pi}=\chi_{\gamma}$. We may also say $\gamma$ is the central character of $X$.

An $\mathbb{H}^{Cl}_n$-module $(\pi,X)$ is said to be quasisimple if any element in $Z(\mathbb{H}^{Cl}_n)$ acts by a scalar. In this case, $\gamma$ defined as above is still called the central character of $X$.

\end{definition}

\subsection{Induced modules} \label{ss def temp mod}

Let us recall a construction of some $\mathbb{H}^{Cl}_{n}$-modules in \cite[Section 4]{HKS}, which is indeed modified from the module of type $A_{n-1}$ in Section \ref{ss steinberg all}. There are also some similar construction of $\mathbb{H}^{Cl}_n$-modules in \cite[Section 4]{Wa}. Fix a partition $\lambda=(n_1,n_2,\ldots, n_r)$ of $n$. Let $S_{\lambda}$ be the subgroup of $S_n$ generated by $s_{i,i+1}$ for $i=\left\{1,\ldots, n \right\} \setminus \left\{ n_1,n_1+n_2,\ldots,n_1+\ldots+n_r \right\}$. It is easy to see that $S_{\lambda}$ is isomorphic to $S_{n_1}\times \ldots \times S_{n_r}$. Let $\mathbb{H}^{Cl}_{\lambda}$ be the super subalgebra of $\mathbb{H}^{Cl}_{n}$ generated by all $w \in S_{\lambda}$, $x_i$ ($i=1,\ldots,n$) and $c_i$ ($i=1,\ldots,n$). Let $\Seg_{\lambda}$ be the super subalgebra of $\mathbb{H}^{Cl}_{\lambda}$ generated by all $w \in S_{\lambda}$ and $c_i$ ($i=1, \ldots,n$). Let $\widetilde{\mathrm{St}}_{\lambda}$ be an  $\mathbb{H}^{Cl}_{\lambda}$module which is identified with $\mathrm{Cl}_n$ as vector spaces and the action of $\mathbb{H}^{Cl}_{\lambda}$ is characterized by:
\begin{eqnarray*}
         c_i. 1 = c_i \quad (i=1,\ldots, n), \quad      s_{\alpha}. 1 = 1 \quad (s_{\alpha} \in S_{\lambda}),
\end{eqnarray*}

\[    x_i. v= \left(\sum_{n_{k-1}+1 \leq j<i \leq n_k} s_{i,j}(1-c_ic_j)\right). v \quad (i= n_k+1,\ldots, n_{k+1}), \]
where $v$ is any vector in $\mathrm{Cl}_n$ and the actions of $s_{i,j}$ and $c_i,c_j$ are the ones defined in (\ref{eqn action st 1}) and (\ref{eqn action st 2}). It is straightforward to check the above actions defines an $\mathbb{H}^{Cl}_{\lambda}$-module by verifying the defining relations of $\mathbb{H}_{\lambda}^{Cl}$. Some details can be found in \cite[Proposition 4.1.1]{HKS}.

\begin{lemma} \label{lem xi on st}
The element $x_i^2$ acts on $\widetilde{\mathrm{St}}_{\lambda}$ by a scalar $(i-n_k-1)(i-n_k)$ where $k=0,\ldots, r-1$ and $i= n_k+1,\ldots, n_{k+1}$  .
\end{lemma}

\begin{proof}
Direct computation or see \cite[Proposition 4.1.1]{HKS}.
\end{proof}
Define the Dirac-type element $D_{\lambda}$ in $\mathbb{H}_{\lambda}^{Cl}$ as:
\[ D_{\lambda} = \sum_{i=1}^n y_i +\sqrt{2}\mathbf{k} \sum_{\alpha \in R_{\lambda}^+} \widetilde{s}_{\alpha} .\]

\begin{proposition} \label{prop Dirac action}
The element $D_{\lambda}$ acts as zero on the $\mathbb{H}^{Cl}_{\lambda}$-module $\widetilde{\mathrm{St}}_{\lambda}$.
\end{proposition}

\begin{proof}
It follows a similar computation of type $A_{n-1}$ in the proof of Proposition \ref{prop vanishing st}.

\end{proof}

Define
\begin{eqnarray} \label{eqn induced mod}
  X_{\lambda} &=& \Ind_{\mathbb{H}^{Cl}_{\lambda}}^{\mathbb{H}^{Cl}_n} \widetilde{\mathrm{St}}_{\lambda} = \mathbb{H}^{Cl}_{n} \otimes_{\mathbb{H}^{Cl}_{\lambda}} \widetilde{\mathrm{St}}_{\lambda}
  \end{eqnarray}
with the map $\pi_{\lambda}$ defining the action of $\mathbb{H}^{Cl}_n$ on $X_{\lambda}$.  Since $Z(\mathbb{H}^{Cl}_{n}) \subset Z(\mathbb{H}^{Cl}_{\lambda})$, any element of $Z(\mathbb{H}^{Cl}_n)$ acts by a scalar on $\widetilde{\mathrm{St}}_{\lambda}$. With the definitions for $X_{\lambda}$ and $Z(\mathbb{H}^{Cl}_{n})$, we have that  $X_{\lambda}$ is quasisimple (Definition \ref{def cc}). The central character of $X_{\lambda}$ can be represented by
\[   (\underbrace{1(1-1), \ldots, n_1(n_1-1)}_{n_1 \mbox{ terms}}, \ldots , \underbrace{1(1-1), \ldots, n_r(n_r-1)}_{n_r \mbox{ terms}} ) \in \mathbb{R}^n. \]
To compute the Dirac cohomology of the above induced modules, we need some more information discussed in the next subsections.

\subsection{$S_n$-structure and $\Seg_n$-structure of ($\pi_{\lambda}$, $X_{\lambda}$)}

We continue to fix a partition $\lambda$ of $n$. Recall that in Definition \ref{def hca}(1) $\mathbb{H}^{Cl}_n$ contains $\mathbb{C}[S_n]$ as a subalgebra. Let $(\pi_V, V=\mathbb{C}^n)$ be the $S_n$-representation such that elements in $S_n$ permute the coordinates.

\begin{lemma} \label{lem induced iso}
 The restriction of $X_{\lambda}$ to $\mathbb{C}[S_n]$ is  isomorphic to
\[ \mathbb{C}[S_n] \otimes_{\mathbb{C}[S_{\lambda}]} \Res_{\mathbb{C}[S_{\lambda}]}^{\mathbb{C}[S_n]} \left(\bigoplus_{i=0}^n \wedge^i V \right), \]
as $\mathbb{C}[S_n]$-modules.
\end{lemma}

\begin{proof}
Note that the restriction of $\widetilde{\mathrm{St}}_{\lambda}$ to $\mathbb{C}[S_{\lambda}]$ is isomorphic to $\Res^{\mathbb{C}[S_n]}_{\mathbb{C}[S_{\lambda}]} (\bigoplus_{i=0}^n \wedge^i V)$. Then ${\mathbb{H}^{Cl}_n} \otimes_{\mathbb{H}^{Cl}_{\lambda}} \widetilde{\mathrm{St}}_{\lambda}$
and  $\mathbb{C}[S_n] \otimes_{\mathbb{C}[S_{\lambda}]} \Res_{\mathbb{C}[S_{\lambda}]}^{\mathbb{C}[S_n]} \left(\bigoplus_{i=0}^n \wedge^i V \right)$ are isomorphic as $\mathbb{C}[S_n]$-modules.

\end{proof}

 It is well-known that we have the following $\mathbb{C}[S_n]$-isomorphism:

\[ \mathbb{C}[S_n] \otimes_{\mathbb{C}[S_{\lambda}]} \Res^{S_n}_{S_{\lambda}}\left(\bigoplus_{i=0}^n \wedge^i V \right) \cong \left(\mathbb{C}[S_n] \otimes_{\mathbb{C}[S_{\lambda}]} \triv \right)\otimes \bigoplus_{i=0}^n \wedge^i V  \]
Here the module in the right hand side is viewed as the tensor product of two $S_n$-representations.
The isomorphism is given by
\[  w \otimes  (v_1 \wedge \ldots \wedge v_i)  \mapsto  (w \otimes 1) \otimes (\pi_V(w)v_1 \wedge \ldots \wedge \pi_V(w)v_i) .\]

Note that the space $\oplus_{i=0}^n \wedge^i V$ can be identified with $\mathrm{Cl}_n$ via the map determined by
\[ e_{i_1} \wedge \ldots \wedge e_{i_r} \mapsto c_{i_1}\ldots c_{i_r} ,\]
where $\left\{ e_1, \ldots, e_n \right\}$ is the standard basis of $V=\mathbb{C}^n$.
Thus $X_{\lambda}=\Ind_{\mathbb{H}_{\lambda}^{Cl}}^{\mathbb{H}^{Cl}_n}\widetilde{\mathrm{St}}_{\lambda}$ can be identified with, as vector spaces, $\left(\mathbb{C}[S_n] \otimes_{\mathbb{C}[S_{\lambda}]} \triv \right)\otimes  U_{\mathrm{Cl}_n}$ via the identification in Lemma \ref{lem induced iso} and the above identification between $\oplus_{i=1}^n \wedge^i V$ and $\mathrm{Cl}_n$. Then if we translate the action of the subalgebra $\Seg_n$ under the above identifications, then we have:
\[   \pi_{\lambda}(w) ( w' \otimes 1 \otimes c_{i_1}\ldots c_{i_r})= ww' \otimes 1 \otimes c_{w(i_1)}\ldots c_{w(i_r)}  ,\]
\[   \pi_{\lambda}(c_i) (w' \otimes 1 \otimes c_{i_1} \ldots c_{i_r}) = w' \otimes 1 \otimes c_{i} c_{i_1} \ldots c_{i_r}   . \]

We have just proven that:
\begin{lemma} \label{lem rest X seg}
As $\Seg_n$-supermodules,
\[ \Res^{\mathbb{H}^{Cl}_n}_{\Seg_n}X_{\lambda}=(\mathbb{C}[S_n] \otimes_{\mathbb{C}[S_{\lambda}]} \triv) \otimes U_{\mathrm{Cl}_n},\]
where the supermodule in the right hand side has the $\Seg_n$-supermodule structure described in Lemma \ref{lem seg mod const}.
\end{lemma}


Recall that $F$ is the functor defined in Section \ref{ss mod relations}.

\begin{proposition} \label{lem F functor res}
As $\Seg_n$-supermodules,
\[\Res^{\mathbb{H}^{Cl}_n}_{\Seg_n}X_{\lambda}=F((\mathbb{C}[S_n] \otimes_{\mathbb{C}[S_{\lambda}]} \triv) \otimes U_{\mathrm{spin}}) ,\]
where $(\mathbb{C}[S_n] \otimes_{\mathbb{C}[S_{\lambda}]} \triv) \otimes U_{\mathrm{spin}}$ has $\mathbb{C}[\widetilde{S}_n]^-$-supermodule described in Lemma \ref{lem spin repns const}.
\end{proposition}

\begin{proof}
By Lemma \ref{lem rest X seg}, it suffices to show
\[ (\mathbb{C}[S_n] \otimes_{\mathbb{C}[S_{\lambda}]} \triv) \otimes U_{\mathrm{Cl}_n} =(\mathbb{C}[S_n] \otimes_{\mathbb{C}[S_{\lambda}]} \triv) \otimes U_{\mathrm{spin}}  \otimes U(n) .
\]
By Lemma \ref{ex cl spin}, there is a $\Seg_n$-module isomorphism $f$ from $U_{\mathrm{Cl}_n}$ to $F(U_{\mathrm{spin}})=U_{\mathrm{spin}} \otimes U(n)$. Then define a vector space isomorphism from $\Seg_n$-module $(\mathbb{C}[S_n] \otimes_{\mathbb{C}[S_{\lambda}]} \triv) \otimes U_{\mathrm{Cl}_n} \rightarrow (\mathbb{C}[S_n] \otimes_{\mathbb{C}[S_{\lambda}]} \triv) \otimes U_{\mathrm{spin}}  \otimes U(n)$ determined by
\[     (w \otimes 1) \otimes (c_{i_1}\ldots c_{i_r} \otimes 1) \mapsto ( w \otimes 1) \otimes f(c_{i_1}\ldots c_{i_r}\otimes 1) . \]
Using the module structure described before Lemma \ref{lem rest X seg}, one can check the linear isomorphism is $\Seg_n$-equivariant.
\end{proof}

\subsection{Hermitian form on ($\pi_{\lambda}$, $X_{\lambda}$)}
We continue to fix a partition $\lambda$ of $n$. In this subsection, we shall construct a Hermitian form on the $\mathbb{H}^{Cl}_n$-module $(\pi_{\lambda}, X_{\lambda})$ such that the adjoint operator of $\pi_{\lambda}(D)$ with respect to such form is $-\pi_{\lambda}(D)$. We will see this makes the computation for the Dirac cohomology $H_D(X)$ of those modules $X$ much easier.



Recall that $\Seg_{\lambda}$ is a subalgebra of $\mathbb{H}^{Cl}_{\lambda}$.
\begin{lemma}
There exists a $\Seg_{\lambda}$-invariant positive definite Hermitian form on $\widetilde{\mathrm{St}}_{\lambda}$.
\end{lemma}

\begin{proof}
Since $\Res^{\mathbb{H}^{Cl}_{\lambda}}_{\Seg_{\lambda}}\widetilde{\mathrm{St}}_{\lambda}=\Res^{\Seg_{n}}_{\Seg_{\lambda}}U_{\mathrm{Cl}_n}$ as $\Seg_{\lambda}$-modules, it suffices to consider the case when $\lambda=(n)$. Recall that $U_{\mathrm{Cl}_n}=\Seg_n \otimes_{\mathbb{C}[S_n]} \triv$ in Section \ref{ss mod relations}. Define $\langle ., .\rangle: U_{\mathrm{Cl}_n} \times U_{\mathrm{Cl}_n} \rightarrow \mathbb{C}$ such that for $1 \leq i_1<\ldots < i_r \leq n$ and $1 \leq j_1<\ldots< j_s \leq n$,
\[   \langle c_{i_1}c_{i_2} \ldots c_{i_r} \otimes 1, c_{j_1}c_{j_2}\ldots c_{j_s} \otimes 1 \rangle =\left\{ \begin{array}{l l} 1 & \mbox{ if $\left\{ i_1,\ldots, i_r \right\} =  \left\{ j_1, \ldots, j_s \right\}$}    \\
$0$ &  \mbox{ otherwise } .
\end{array} \right.
\]
It is straightforward to check $\langle , \rangle$ satisfies the desired properties.
\end{proof}



We denote the $\Seg_{\lambda}$-invariant Hermitian form on $ \widetilde{\mathrm{St}}_{\lambda}$ in the above lemma by $\langle .,. \rangle_{\lambda}$. Recall that $ X_{\lambda}=\mathbb{H}^{Cl}_n \otimes_{\mathbb{H}^{Cl}_{\lambda}} \widetilde{\mathrm{St}}_{\lambda}$. We define a bilinear form $\langle .,. \rangle$ on $X_{\lambda}$ characterized by:
\begin{eqnarray*}
\langle w_1 \otimes v_1, w_2 \otimes v_2 \rangle &= \delta_{w_1S_{\lambda}, w_2S_{\lambda}} \langle \pi_{\lambda}(w_2^{-1}w_1) v_1 , v_2 \rangle_{\lambda}
 \end{eqnarray*}
where $w_1,w_2 \in S_n$ and $\delta_{w_1S_{\lambda}, w_2S_{\lambda}}=1$ if $w_1S_{\lambda}=w_2S_{\lambda}$ and $\delta_{w_1S_{\lambda}, w_2S_{\lambda}}=0$ otherwise.

\begin{lemma} \label{lem positive definite}
$\langle . , . \rangle$ defined above is a positive definite Hermitian form.
\end{lemma}

\begin{proof}
This follows from the property that $\langle .,.\rangle_{\lambda}$ is positive definite and Hermitian.
\end{proof}

We next compute the adjoint operator of $\pi_{\lambda}(D)$ with respect to $\langle .,. \rangle$. We begin with some lemmas.
\begin{lemma} \label{lem comp D lam}
For $v_1, v_2 \in  \widetilde{\mathrm{St}}_{\lambda}$, $\langle D \otimes v_1, 1 \otimes v_2 \rangle = \langle 1 \otimes v_1, D \otimes v_2 \rangle=0$.
\end{lemma}

\begin{proof}
For $\alpha \in R^+ \setminus R_{\lambda}^+$, one has
\[   \langle \widetilde{s}_{\alpha} \otimes v_1, 1 \otimes v_2 \rangle = 0. \]
With the equality
\[  D = D_{\lambda} +\sqrt{2}\mathbf{k} \sum_{\alpha >0, \alpha \in R^+\setminus R_{\lambda}^+} \widetilde{s}_{\alpha}  ,\]
one has $\langle D \otimes v_1, 1 \otimes v_2 \rangle = \langle D_{\lambda}.v_1, v_2 \rangle_{\lambda}$. Then we have $\langle D \otimes v_1, 1 \otimes v_2 \rangle=0$ by Proposition \ref{prop Dirac action}. The proof for $\langle 1 \otimes v_1, D \otimes v_2 \rangle=0$ is similar.
\end{proof}

\begin{lemma} \label{lem weyl group}
Suppose $\beta_1 \neq \beta_2$ and $\beta_1, \beta_2 \in R^+ \setminus R_{\lambda}^+$. Then $s_{\beta_1}s_{\beta_2} \not\in S_{\lambda}$.
\end{lemma}

\begin{proof}
In the following, we implicitly use several times the fact that any element in $S_{\lambda}$ cannot send a positive root not in $R_{\lambda}$ to a negative root. If $\langle \beta_1, \beta_2 \rangle =0$, then $s_{\beta_1}s_{\beta_2}(\beta_2)=-\beta_2<0$. Since $\beta_2 \not\in R_{\lambda}$, $s_{\beta_1}s_{\beta_2} \notin S_{\lambda}$. If $\langle \beta_1, \beta_2 \rangle =-1$, then $s_{\beta_2}(\beta_1)=\beta_1 +\beta_2 >0$. Moreover, $s_{\beta_1}s_{\beta_2}(s_{\beta_2}(\beta_1))=-\beta_1<0$. Since $\beta_1+\beta_2 \notin R_{\lambda}$, $s_{\beta_1}s_{\beta_2} \notin S_{\lambda}$. If $\langle \beta_1, \beta_2 \rangle =1$, then either $s_{\beta_1}(\beta_2) >0$ or $s_{\beta_2}(\beta_1) >0$. In the case that $s_{\beta_1}(\beta_2) >0$, $s_{\beta_1}s_{\beta_2}(\beta_2)=-s_{\beta_1}(\beta_2)<0$. Then since $\beta_2 \notin R_{\lambda}$, $s_{\beta_1}s_{\beta_2} \notin S_{\lambda}$. Similar argument by considering $(s_{\beta_1}s_{\beta_2})^{-1}$ can prove another case.
\end{proof}

\begin{proposition} \label{lem adj op}
The adjoint operator of $\pi_{\lambda}(D)$ with respect to $\langle .,.\rangle$ is $-\pi_{\lambda}(D)$.
\end{proposition}

\begin{proof}
It suffices to show that
\[  \langle Dw_1 \otimes v_1 , w_2 \otimes v_2 \rangle = \langle w_1 \otimes v_1, -Dw_2 \otimes v_2 \rangle \]
for any $w_1, w_2 \in S_n$ and $v_1, v_2 \in X_{\lambda}$. To this end, we consider two cases.
Suppose $w_1S_{\lambda}=w_2S_{\lambda}$. Then,
\begin{eqnarray*}
\langle Dw_1\otimes v_1, w_2 \otimes v_2 \rangle & =& \langle w_2^{-1}Dw_1 \otimes v_1, 1 \otimes v_2 \rangle \\
                                                     &=& \langle D w_2^{-1}w_1 \otimes v_1, 1 \otimes v_2 \rangle \\
                                                     &=& \langle D\otimes (w_2^{-1} w_1).v_1, 1 \otimes v_2 \rangle \\
																										 &=& 0 \quad \mbox{ (by Lemma \ref{lem comp D lam})}
\end{eqnarray*}
Similarly, we also have
\[ \langle w_1\otimes v_1,D w_2 \otimes v_2 \rangle =0. \]
and so $ \langle Dw_1 \otimes v_1 , w_2 \otimes v_2 \rangle = \langle w_1 \otimes v_1, -Dw_2 \otimes v_2 \rangle$.

Now we suppose that $w_1S_{\lambda} \neq w_2S_{\lambda}$. Without loss of generality, assume that $w_2^{-1}w_1$ is a minimal representative in $w_2^{-1}w_1S_{\lambda}$.
\begin{eqnarray*}
& &\langle w_2^{-1}w_1D\otimes v_1, 1 \otimes v_2  \rangle  \\
&=& \langle w_2^{-1}w_1\sqrt{2}\sum_{\alpha >0}\mathbf{k} \widetilde{s}_{\alpha} \otimes v_1, 1\otimes v_2 \rangle \\
&=& \langle 1 \otimes v_1 ,\sqrt{2}\sum_{\alpha >0}\mathbf{k} \widetilde{s}_{\alpha} w_1^{-1} w_2 \otimes v_2 \rangle   \\
&=& -\langle 1 \otimes v_1, D w_1^{-1}w_2 \otimes v_2 \rangle+\langle 1 \otimes v_1, w_1^{-1}w_2D\otimes v_2\rangle+\langle 1 \otimes v_1 , \sqrt{2}\sum_{\alpha >0} \mathbf{k} \widetilde{s}_{\alpha} w_1^{-1} w_2 \otimes v_2 \rangle  \\
\end{eqnarray*}

It remains to show
\[ \langle 1 \otimes v_1, w_1^{-1}w_2D\otimes v_2\rangle+\sqrt{2}\langle 1 \otimes v_1 , \sum_{\alpha >0}\mathbf{k} \widetilde{s}_{\alpha} w_1^{-1} w_2 \otimes v_2 \rangle =0. \]
By Lemma \ref{lem weyl group}, there exists at most one $\beta \in R^+ \setminus R_{\lambda}$ such that $w_1^{-1}w_2s_{\beta}\in S_{\lambda}$. If such $\beta$ does not exist, then the two terms in the left hand side of the above equation are both zero and so the equation holds. If such unique $\beta$ exists, let $\beta'=-w_1^{-1}w_2(\beta)$. Note that $\beta'>0$ otherwise $w_1^{-1}w_2s_{\beta}\not\in S_{\lambda}$. Then

\begin{eqnarray*}
 & & \langle 1 \otimes v_1, w_1^{-1}w_2D\otimes v_2\rangle+\sqrt{2}\langle 1 \otimes v_1 , \sum_{\alpha >0} \mathbf{k}\widetilde{s}_{\alpha} w_1^{-1} w_2 \otimes v_2 \rangle \\
 &=&  \sqrt{2}\mathbf{k} \langle 1 \otimes v_1, w_1^{-1}w_2\widetilde{s}_{\beta}\otimes v_2\rangle+\sqrt{2}\mathbf{k}\langle 1 \otimes v_1 , \widetilde{s}_{\beta'} w_1^{-1} w_2 \otimes v_2 \rangle \mbox{ (by definition of $D$ and $\langle.,.\rangle$) }\\
 &=& -\sqrt{2}\mathbf{k} \langle 1\otimes v_1, \widetilde{s}_{\beta'}w_1^{-1}w_2 \otimes v_2 \rangle +\sqrt{2}\mathbf{k}\langle 1 \otimes v_1, \widetilde{s}_{\beta'}w_1^{-1}w_2 \otimes v_2 \rangle \quad \mbox{ (by Lemma \ref{lem simple rel}(4)) }\\
 &=& 0
\end{eqnarray*}
This completes the proof.
\end{proof}

\begin{proposition} \label{prop reduced dir coh}
Let $(\pi_{\lambda}, X_{\lambda})$ be the $\mathbb{H}_n^{Cl}$-module as in (\ref{eqn induced mod}). Then
\[   \ker \pi_{\lambda}(D) = \ker \pi_{\lambda}(D^2)  \]
and
\[ \ker \pi_{\lambda}(D) \cap \im\pi_{\lambda}(D) = 0 .\]
In particular, $H_D(X_{\lambda})=\ker \pi_{\lambda}(D^2)$.
\end{proposition}

\begin{proof}
It is clear that $\ker \pi_{\lambda}(D) \subset \ker \pi_{\lambda}(D^2)$. For $ v \in \ker \pi_{\lambda}(D^2)$, $\langle \pi_{\lambda}(D)v,-\pi_{\lambda}(D)v \rangle=\langle \pi_{\lambda}(D^2)v,v \rangle=0$ by Proposition \ref{lem adj op}. Since $\langle . , .\rangle$ is positive definite by Lemma \ref{lem positive definite}, $\pi_{\lambda}(D)v=0$. This proves the first equation $\ker \pi_{\lambda}(D) = \ker \pi_{\lambda}(D^2)$. The equation $\ker \pi_{\lambda}(D) \cap \im\pi_{\lambda}(D) = 0$ follows from the first one.
\end{proof}

\subsection{Dirac cohomology of $X_{\lambda}$}

Let $\mathcal P_n$ be the set of partitions of $n$. One can attach an element in $\mathcal P_n$ to a point in $\mathbb{R}^n$ via the Jacobson-Morozov triple. The map, denoted $\Phi_1: \mathcal P_n \rightarrow \mathbb{R}^n$ can be explicitly described as:
\[    (n_1, n_2, \ldots, n_r) \mapsto (\underbrace{-n_1+1, -n_1+3,\ldots, n_1-1}_{n_1 \mbox{ terms }}, \ldots, \underbrace{-n_r+1, -n_r+3,\ldots, n_r-1}_{ n_r \mbox{ terms }} )  .
\]

There is another way to attach an element in $\mathcal P_n$ to a point in $\mathbb{R}^n$ via the central characters of the modules $X_{\lambda}$. This map, denoted $\Phi_2: \mathcal P_n \rightarrow \mathbb{R}^n$ is:
\[  (n_1, n_2, \ldots, n_r) \mapsto (\underbrace{\sqrt{(1-1)1}, \ldots, \sqrt{(n_1-1)n_1}}_{n_1 \mbox{ terms }}, \ldots, \underbrace{\sqrt{(1-1)1}, \ldots,\sqrt{n_r(n_r-1)}}_{n_r \mbox{ terms }} )  .\]

The first interesting computational fact is the following:
\begin{lemma} \label{lem eql spin}
For a partition $\lambda$ of $n$, $|\Phi_1(\lambda)|=|\Phi_2(\lambda)|$, where $|.|$ denotes the standard Euclidean norm in $\mathbb{R}^n$.
\end{lemma}

\begin{proof}
This follows from the computation that
\[  \sum_{k=1}^{n_i} \left( -n_i+2k-1 \right)^2 = \sum_{k=1}^{n_i}k(k-1)=\frac{1}{3}(n_i-1)n_i(n_i+1).   \]
\end{proof}

For each $\lambda \in \mathcal P_n$, define a $S_n$-representation:
\[  W_{\lambda} = (\Ind_{\mathbb{C}[S_{\lambda}]}^{\mathbb{C}[S_n]}\triv ) \cap (\Ind_{\mathbb{C}[S_{{\lambda}^t}]}^{\mathbb{C}[S_n]} \sgn) ,\]
where $\sgn$ and $\triv$ are respectively the sign and trivial representations of $S_{\lambda}$, and $\lambda^t$ is the conjugate of $\lambda$. It is well-known that $W_{\lambda}$ exhausts the list of irreducible representations of $S_n$.

Define
\[ \Omega_{\mathbb{C}[\widetilde{S}_n]^-}=2\mathbf{k}^2\sum_{\alpha>0, \beta>0, s_{\alpha}(\beta)<0} \widetilde{t}_{\alpha}\widetilde{t}_{\beta} \in \mathbb{C}[\widetilde{S}_n]^-. \]

Let $P_n^{\mathrm{dist}}$ be the set of partitions of $n$ with distinct parts. Recall that $\Irr_{\mathrm{sup}}\mathbb{C}[\widetilde{S}_n]^-$ (resp. $\Irr_{\mathrm{sup}} \Seg_n$) is the set of irreducible supermodules of $\mathbb{C}[\widetilde{S}_n]^-$ (resp. $\Seg_n$). Recall that the equivalence relation $\sim_{\Pi}$ on $\Irr_{\mathrm{sup}}\mathbb{C}[\widetilde{S}_n]^-$ or $\Irr_{\mathrm{sup}}\Seg_n$ is defined in Section \ref{ss rel sup and ord}.

\begin{proposition} \cite[Part of Theorem 1.0.1]{Ci} (also see \cite{St}) \label{prop ciu bij}
There exists a bijection $\Psi_1: \mathcal P_n^{\mathrm{dist}} \rightarrow \Irr_{\mathrm{sup}}\mathbb{C}[\widetilde{S}_n]^-/\sim_{\Pi}$ such that for each partition $\lambda$ of $n$, there exists a representative $(\sigma, U) \in \Psi_1(\lambda)$ with the properties that
\[  \mathbf{k}^2|\Phi_1(\lambda)|^2=\chi_{\sigma}(\Omega_{\mathbb{C}[\widetilde{S}_n]^-}) \]
and
\[  \Hom_{\mathbb{C}[\widetilde{S}_n]^-}(U, W_{\lambda} \otimes U_{\mathrm{spin}}) \neq 0 .
\]
\end{proposition}
\begin{proof}
In  \cite[Theorem 1.0.1]{Ci}, the set $\Irr \mathbb{C}[\widetilde{S}_n]^-/\sim_{\sgn}$ is considered instead of $\Irr_{\mathrm{sup}} \mathbb{C}[\widetilde{S}_n]^-/\sim_{\Pi}$. By Proposition \ref{prop cri irr Sn}, there is a natural bijection between $\Irr \mathbb{C}[\widetilde{S}_n]^-/\sim_{\sgn}$ and $\Irr_{\mathrm{sup}} \mathbb{C}[\widetilde{S}_n]^-/\sim_{\Pi}$. Then one can now apply \cite[Theorem 1.0.1]{Ci}.

\end{proof}

Here is an analogue of Proposition \ref{prop ciu bij}. Recall that $\Omega_{\Seg_n}$ (i.e. $\Omega_{\mathrm{Seg}(W_{A_{n-1}})}$) is defined in Theorem \ref{thm d sq}.

\begin{proposition} \label{prop ker sq}
There exists a bijection $\Psi_2: \mathcal P_n^{\mathrm{dist}} \rightarrow \Irr_{\mathrm{sup}}\Seg_n/\sim_{\Pi}$ such that  there exists a representative $(\sigma, U) \in \Psi_2(\lambda)$ with the properties that
\[ \mathbf{k}^2|\Phi_2(\lambda) |^2  =\chi_{\sigma}(\Omega_{\Seg_n})  \]
and
\[  \Hom_{\Seg_n}(U, F( W_{\lambda} \otimes U_{\mathrm{spin}})   ) \neq 0 .
\]
\end{proposition}

\begin{proof}
Note that for an irreducible $\mathbb{C}[\widetilde{S}_n]^-$-supermodule $U$, $F(U)$ is either an irreducible supermodule or the direct sum of two irreducible supermodules of opposite grading. Thus we could define $\Psi_2(\lambda)$ to be the unique equivalence class in $\Irr_{\mathrm{sup}}\Seg_n/\sim_{\Pi}$ containing the irreducible supermodule(s) in $F(U)$ for a representative $U \in \Phi_1(\lambda)$, where $\Phi_1$ is defined in Proposition \ref{prop ciu bij}.

It remains to check those two properties. Recall $F(U)=U \otimes U(n)$ and the action of $\Seg_n$ on $F(U)$ is defined in Section \ref{ss mod relations}. Then for $u \otimes u' \in U \otimes U(n)$,
\begin{eqnarray*}
\Omega_{\Seg_n}.(u \otimes u') &=& 2\mathbf{k}^2\sum_{\alpha, \beta >0, s_{\alpha}(\beta)<0}\widetilde{s}_{\alpha}\widetilde{s}_{\beta}.(u \otimes u') \\
&=& 2\mathbf{k}^2(\sum_{\alpha, \beta>0,s_{\alpha}(\beta)<0} \widetilde{t}_{\alpha}\widetilde{t}_{\beta}.u) \otimes u'   \\
&=& \chi_{\sigma}(\Omega_{\mathbb{C}[\widetilde{S}_n]^-}) u \otimes u'
\end{eqnarray*}
Thus for any irreducible supermodule $(\sigma', U')$ in $F(U)$, $\chi_{\sigma'}(\Omega_{\Seg_n})=\chi_{\sigma}(\Omega_{\mathbb{C}[\widetilde{S}_n]^-})$. Then combining with Lemma \ref{lem eql spin} and Proposition \ref{prop ciu bij}, we have shown the first property.

The second property follows from
\[  \Hom_{\Seg_n}(F(U'), F( W_{\lambda} \otimes U_{\mathrm{spin}})   )=\Hom_{\mathbb{C}[\widetilde{S}_n]^-}(U', G\circ F( W_{\lambda} \otimes U_{\mathrm{spin}})   ) \neq 0 ,
\]
where the last equality follows from Propositions \ref{prop adj functor} and \ref{prop ciu bij}.
\end{proof}

\begin{lemma} \label{lem hom nonzero}
For a partition $\lambda$ of $n$ with distinct parts, there exists a representative $U \in \Phi_2(\lambda)$ such that
\[ \Hom_{\Seg_n}(  U, \Res_{\Seg_n}^{\mathbb{H}_{n}^{Cl}}X_{\lambda} ) \neq 0. \]
\end{lemma}

\begin{proof}
This follows from
\begin{eqnarray*}
& &\Hom_{\Seg_n}(  U, \Res_{\Seg_n}^{\mathbb{H}_{n}^{Cl}}X_{\lambda} ) \\
&=& \Hom_{\Seg_n}(U, F((\mathbb{C}[S_n] \otimes_{\mathbb{C}[S_{\lambda}]} \triv) \otimes U_{\mathrm{spin}})) \quad (\mbox{by Proposition \ref{lem F functor res}}) \\
&\supseteq& \Hom_{\Seg_n}(U, F(W_{\lambda} \otimes U_{\mathrm{spin}}))  \quad (\mbox{by definition of $W_{\lambda}$})
\end{eqnarray*}
The statement now follows from Proposition \ref{prop ker sq}.
\end{proof}

The following theorem states that the induced modules $(\pi_{\lambda}, X_{\lambda})$ with $\lambda$ of distinct parts have non-zero Dirac cohomologies.

\begin{theorem} \label{thm vog examp}
Let $\lambda$ be a partition of $n$ with distinct parts. Let $(\pi_{\lambda}, X_{\lambda})$ be the $\mathbb{H}^{Cl}_{n}$-module defined in (\ref{eqn induced mod}). Let $\Psi_2$ be the map defined in Proposition \ref{prop ker sq}. Then there exists a representative $U$ in $\Psi_2(\lambda)$ such that
\[  \Hom_{\Seg_n}(U, H_D(X_{\lambda})) \neq 0.     \]
In particular, $H_D(X_{\lambda})$ is non-zero.
\end{theorem}

\begin{proof}
For a fixed $\lambda \in \mathcal P^{\mathrm dist}_n$, let $U$ be a $\Seg_n$-module with the property in Lemma \ref{lem hom nonzero}. Then there exists a non-zero vector $v$ in the isotypical component $U$ of $X_{\lambda}$. By Theorem \ref{thm d sq}, Lemma \ref{lem xi on st} and Proposition \ref{prop ker sq}, $\pi_{\lambda}(D^2)v=(\chi_{\pi_{\lambda}}(\Omega_{\mathbb{H}^{Cl}_n})-\chi_{\Psi_2(\lambda)}(\Omega_{\Seg_n}))v=(\mathbf{k}^2|\Phi_2(\lambda)|^2-\chi_{\Psi_2(\lambda)}(\Omega_{\Seg_n}))v=0$. Hence, $v \in \ker(\pi_{\lambda}(D^2))$. By Proposition \ref{prop reduced dir coh}, $v \in  H_D(X_{\lambda})=\ker\pi_{\lambda}(D^2)$. This proves the theorem.

\end{proof}

The Dirac cohomology $H_D(X_{\lambda})$ also provides a way to realize irreducible $\Seg_n$-supermodules.

\begin{corollary} \label{cor image of HD}
For each $\lambda \in \mathcal P_n^{\mathrm{dist}}$, there exists a unique irreducible $\Seg_n$-supermodule $U$, up to the equivalence of $\sim_{\Pi}$, such that  $\Hom_{\Seg_n}(U,H_D(X_{\lambda}))\neq 0$. Let $[H_D(X_{\lambda})]$ be an irreducible submodule of $H_D(X_{\lambda})$. Then
\[    \Irr_{\mathrm{sup}} \Seg_n =\bigsqcup_{\lambda \in \mathcal P^{\mathrm{dist}}_n} \left\{ [H_D(X_{\lambda})], \Pi([H_D(X_{\lambda})]) \right\},
\]
where $\bigsqcup$ means the disjoint union.
\end{corollary}

\begin{proof}
For the first assertion, the existence has been proved in  Theorem \ref{thm vog examp} and we only have to prove the uniqueness. Let $(\sigma',U')$ be  an irreducible $\Seg_n$-module such that
\[ \Hom_{\Seg_n}(U', H_D(X_{\lambda})) \neq 0 .\] Then $\chi_{\pi_{\lambda}}=\chi^{\sigma'}$ by Theorem \ref{thm Vogan conjecture HC} and Theorem \ref{thm vog examp}. On the other hand, by Proposition \ref{prop ker sq}, $(\sigma',U')$ is in $\Phi_2(\lambda')$ for some $\lambda' \in \mathcal \mathcal P_n^{\mathrm{dist}}$. Then
\[\Hom_{\Seg_n}(U', H_D(X_{\lambda'})) \neq 0 \]
 and by Theorem \ref{thm Vogan conjecture HC} again,
$\chi_{\pi_{\lambda'}}=\chi^{\sigma'}$. Thus $\chi_{\pi_{\lambda'}}=\chi_{\pi_{\lambda}}$ and so $\lambda=\lambda'$. This implies the uniqueness.

The second assertion follows from the first assertion and the bijectivity of $\Phi_2$ in Proposition \ref{prop ker sq}.
\end{proof}

Let $K(\mathbb{H}^{Cl}_n)$ (resp. $K(\Seg_n)$)  be the Grothendieck group of finite-dimensional $\mathbb{H}^{Cl}_n$-supermodules (resp. finite-dimensional $\Seg_n$-supermodules). Then the Dirac cohomology $H_D$ induces a map, still denoted $H_D$, from $K(\mathbb{H}^{Cl}_n)$ to $K(\Seg_n)$. Corollary \ref{cor image of HD} implies the following:

\begin{corollary}
The image of $H_D: K(\mathbb{H}^{Cl}_n) \rightarrow K(\Seg_n)$ has finite index in $K(\Seg_n) $
\end{corollary}

Recall that the algebra homomorphism $\zeta:  Z(\mathbb H_n^{Cl}) \rightarrow Z(\Seg_n)_0$ is defined in Theorem \ref{thm rel centers}. We also have:

\begin{corollary} \label{cor zeta onto}
The map $\zeta:  Z(\mathbb H_n^{Cl}) \rightarrow Z(\Seg_n)_0$ is surjective.
\end{corollary}

\begin{proof}
It suffices to show  $\dim(\im \zeta) \geq \dim Z(\Seg_n)_0$. By Theorem \ref{thm Vogan conjecture HC} and Theorem \ref{thm vog examp}, for any partition $\lambda \in \mathcal P^{\mathrm{dist}}_n$, there exists $(\sigma_{\lambda}, U_{\lambda}) \in \Irr_{\mathrm{sup}}\Seg_n$,  such that $\chi_{\pi_{\lambda}}=\chi^{\sigma_{\lambda}}$. Since the central characters $\left\{ \chi_{\pi_{\lambda}}  \right\}_{\lambda \in \mathcal P^{\mathrm{dist}}_n}$ are  linearly independent over $\mathbb{C}$, $\left\{\chi^{\sigma_{\lambda}} \right\}_{\lambda \in \mathcal P^{\mathrm{dist}}_n}$ are also linearly independent. Then we have $\dim (\im \zeta)$ is not less than the cardinality of $\mathcal P^{\mathrm{dist}}_n$. Now the statement follows from the fact that $\dim Z(\Seg_n)_0$ is equal to the cardinality of $\Irr_{\mathrm{sup}}(\Seg_n)/\sim_{\Pi}$, which is the same as the cardinality of $\mathcal P^{\mathrm{dist}}_n$.

\end{proof}

\begin{remark} \label{rmk surjective homo}
The author would like to thank Professor Weiqiang Wang for pointing out that there is a canonical surjective superalgebra morphism from $\mathbb{H}^{Cl}_n$ to $\Seg_n$ \cite[Remark 15.4.7]{Kl}. Denote the map to be $\zeta'$.  According to \cite[Remark 15.4.7]{Kl}, the map $\zeta'$ sends $x_i$ to the Jucys-Murphy type element
\[ \zeta'(x_i)= \sum_{1 \leq j<i } s_{i,j}(1-c_ic_j) ,\]
and $\zeta'$ is an identity on $\Seg_n$. It is straightforward to check $\zeta'(D)=0$. By considering
\[   z = \zeta(z)+Dh+hD \]
and applying $\zeta'$ on both sides, $\zeta'(z)=\zeta(z)$. Hence $\zeta'$ agrees with $\zeta$ on $Z(\mathbb{H}^{Cl}_n)$. The author would like to thank one of the referees to point out that the map $\zeta'$ has already been proven to be surjective in \cite{Ru} as a special case. This in turn gives another way to see $\zeta$ is surjective.
\end{remark}

\end{document}